\documentclass[11pt, reqno]{article}

\usepackage[utf8]{inputenc}
\usepackage[T1]{fontenc}

\usepackage{setspace}
\usepackage[margin=1.25in]{geometry}
\usepackage{parskip}
\usepackage[sc]{mathpazo}
\usepackage[T1]{eulervm}
\usepackage[scaled]{helvet}

\usepackage{amsmath,amsthm, amssymb, amsfonts}
\usepackage{mathtools}
\usepackage{graphicx}
\usepackage[hidelinks]{hyperref}

\theoremstyle{plain}
\newtheorem{thm}{Theorem}
\newtheorem{cor}{Corollary}[section]
\newtheorem{lem}{Lemma}[section]
\newtheorem{prop}{Proposition}[section]

\theoremstyle{definition}

\newtheorem{rem}{Remark}[section]

\numberwithin{equation}{section}

\newcommand{\E}[1]{\mathbf{E}\left[#1\right]}

\newcommand{\ind}[1]{\mathbf{1}_{\{ #1 \}}}
\newcommand{\dt}[1]{\mathrm{det}\left(#1\right)}
\newcommand{\R}{\mathbb{R}}

\title{The extreme statistics of some noncolliding Brownian processes}
\author{Mustazee Rahman\thanks{\textsc{Department of Mathematical Sciences, Durham University}. \textit{Email}: \texttt{mustazee@gmail.com}}}
\date{}

\begin{document}

\maketitle

\begin{abstract}
    \noindent We consider certain noncolliding interacting particle systems driven by Brownian noise.
    A key example is drifted Brownian motions conditioned not to intersect and related models
    of eigenvalues of Hermitian random matrices. We establish limit theorems for the extremal
    particle. We find: (i) the scaling limit of the largest eigenvalue of Brownian motion over
    Hermitian, positive-definite matrices, (ii) Airy process limit for the largest eigenvalue
    of Dyson's Brownian motion for GUE started from generic initial conditions, and (iii) a
    Fredholm determinant formula for the maximum of the top path among noncolliding
    Brownian bridges and, as a byproduct, a new formula for the law of largest eigenvalue in a particular
    Laguerre Orthogonal Ensemble as well as for a related point-to-line last passage percolation model.
\end{abstract}

\section{Introduction}
A noncolliding Brownian process, loosely speaking, is a continuous stochastic process of particles diffusing so as to repel each other. Perhaps the most notable example is Dyson's Brownian motion, which describes the eigenvalues of certain matrix valued diffusions \cite{Dys}. This is equivalent, in the GUE case, to Brownian motions conditioned not to intersect via Doob's $h$-transform \cite{Dys, Gra}. Another example is an exclusion process, such as the totally asymmetric simple exclusion process and its Brownian analogue \cite{NQR,WFS}. These two examples are related; they can be realized as projections of dynamics on certain two-dimensional point processes, namely on Gelfand-Tsetlin patterns \cite{AOCW, BF, War}. The latter turn out to be determinantal point processes, thereby leading to determinant formulas for the laws of the former. In this article we establish limit theorems for the laws of the extremal particle in a class of noncolliding processes by way of such formulas.

The article has three parts, all of them tied together by a model of Brownian last passage percolation with drifts and a boundary. A similar model was first studied in \cite{Rah}.

In the first part, we consider the largest eigenvalue of a random matrix which consists of a matrix with equidistant real eigenvalues perturbed by a GUE matrix (studied in \cite{For, GP, Jo2, OCJ}). We derive the scaling limit of its largest eigenvalue as the matrix dimensions go to infinity, resulting in what appears to be a new probability distribution. This model is also closely related to Brownian motion in the symmetric space $GL(n, \mathbb{C})/U(n)$, which may be identified with Brownian motion over the space of Hermitian, positive-definite matrices (see \cite{Bia, OCJ}).

In the second part, we consider Dyson's Brownian motion for GUE started from general initial conditions.
For a suitable class of initial conditions, we prove that the scaling limit of the largest particle (largest eigenvalue) converges to the Airy process.
Claeys, Neuschel and Venker \cite{CNV} proves such a limit theorem and below we compare it with our result.
This is an universality result, as the Airy process governs extremal fluctuations in many models such as eigenvalues of random matrices \cite{CNV, EY, FHN}, random interface growth models in the KPZ universality class \cite{ACH, CH,Jo5,MQR,PS} and random tilings \cite{AH,Gor,Jo4}.

In the third part, we consider Hermitian Brownian motion over $n \times n$ matrices with a drift, and the largest eigenvalue of this process. Fitzgerald and Warren \cite{FW} proved that the all time maximum of the largest eigenvalue is given by a point-to-line last passage percolation value. We provide a Fredholm determinant formula for the law of this point-to-line last passage value. As a corollary, based on a connection to noncolliding Brownian bridges due to Nguyen and Remenik \cite{NR}, we derive a Fredholm determinant formula for the law of the largest eigenvalue of a matrix from the Laguerre Orthogonal Ensemble.

The rest of the Introduction will elaborate on these parts and present the main results.

\subsection{The largest eigenvalue in a random matrix model} \label{sec:arithmeticmodel}

An $n \times n$ matrix from the Gaussian Unitary Ensemble (GUE) takes the form
$$ H^{\mathrm{GUE}} = \frac{W + W^*}{\sqrt{2}}$$
where $W$ is an $n \times n$ matrix with i.i.d.~entries consisting of standard complex Gaussian random variables
($W_{i,j} = \frac{W^1_{i,j} + \mathbf{i}W^2_{i,j}}{\sqrt{2}}$ where $W^k_{i,j}$ are independent, standard real Gaussian random variables).
Consider the random matrix
\begin{equation} \label{eqn:Hreg}
    H(\tau) = H^{\mathrm{reg}} + \sqrt{\tau} \, H^{\mathrm{GUE}}, \quad \tau > 0,
\end{equation}
where $H^{\mathrm{reg}}$ is a (deterministic) Hermitian matrix whose eigenvalues are ``structured".
Think of \eqref{eqn:Hreg} as a matrix model whereby a structured matrix is perturbed by a small amount of disorder from the GUE.
What affect does the perturbation leave on the eigenvalues as $n$ tends to infinity?

This is of course a rather general question. Suppose the eigenvalues of $H^{\mathrm{reg}}$ are equidistant on the real line, forming an arithmetic progression.
Then \eqref{eqn:Hreg} is relevant to some questions in nuclear physics, being a toy
model for a perturbation of the harmonic oscillator; see \cite{For, GP} for motivation.
In \cite{Jo1} Johansson derives a formula for the eigenvalue distribution of $H(\tau)$ in terms of the eigenvalues of $H^{\mathrm{reg}}$;
see also \cite{BH1,BH2, Shc}. This formula has been useful to understand the universal behaviour of the eigenvalues of random Hermitian matrices; see \cite{EY} for a recent survey.
The article \cite{Jo1} establishes the universal behaviour of the eigenvalues of $H(\tau)$ in the bulk of the spectrum in the case
where $H^{\mathrm{reg}}$ itself is a random Hermitian matrix (a Wigner matrix).

In \cite{Jo2} the bulk eigenvalues of \eqref{eqn:Hreg} is investigated
assuming $H^{\mathrm{reg}}$ has equidistant eigenvalues, and an interesting correlation kernel is found in the large $n$ limit.
We are interested in the largest eigenvalue $\lambda_{max}(H(\tau))$ of $H(\tau)$ in \eqref{eqn:Hreg} assuming that the spectrum
of $H^{\mathrm{reg}}$ has equidistant real eigenvalues. If $\tau > 0$ remains fixed, then we may reduce to the case $\tau = 1$
because $H^{\mathrm{reg}}/\sqrt{\tau}$ will also have spectrum following an arithmetic progression. We find the following limit theorem.

\begin{thm} \label{thm:arithmetic}
Consider, for each $n$, the model \eqref{eqn:Hreg} and assume the eigenvalues of $H^{\mathrm{reg}} = H^{\mathrm{reg}}_n$ are
$\lambda^n_i = \lambda_1^n - \Delta \cdot (i-1)$ for $1 \leq i \leq n$ with a fixed $\Delta > 0$.
Define the rescaled random variable
$$ \lambda_{max}^n = \Delta \cdot (\lambda_{max}(H(1)) - \lambda_1^n) - \log(n-1).$$
Then, for each $a \in \R$,
$$ \lim_n \mathbf{Pr}(\lambda^n_{max} \leq a) = \dt{I - K_{\Delta}}_{L^2[a,\infty)}$$
where the integral kernel $K_{\Delta}$ is as follows.
$$K_{\Delta}(x,y) = \frac{1}{(2 \pi \mathbf{i})^2} \oint_{\gamma_{rec}}d\zeta \oint_{\gamma_{ver}} dz\,
    \frac{e^{\frac{\Delta^2}{2} z^2 -yz}}{e^{\frac{\Delta^2}{2}\zeta^2-x\zeta}} \cdot \frac{\Gamma(\zeta)}{\Gamma(z)}\cdot \frac{1}{z-\zeta}.$$
Here $\gamma_{ver}$ is the vertical contour $\{\Re(z) = 1\}$ oriented upwards, and $\gamma_{rec}$ is the counter clockwise oriented contour
$\{t \pm \mathbf{i}/2; t \leq 1/2\} \cup \{1/2 + \mathbf{i}t; |t| \leq 1/2\}$. Also, $\Gamma(z)$ is the Gamma function.
\end{thm}

\begin{rem}
    The map $a \mapsto \dt{I-K_{\Delta}}_{L^2[a,\infty)}$ should be the c.d.f.~of a probability measure although it is not trivial to
    verify this, in particular, to show that $\dt{I-K_{\Delta}}_{L^2[a,\infty)} \to 0$ as $a \to -\infty$.
\end{rem}

\subsubsection{Brownian motion over Hermitian, positive-definite matrices}
Let $GL(n, \mathbb{C})$ be the group of $n \times n$ invertible matrices with entries from $\mathbb{C}$.
A left-invariant Brownian motion on $GL(n,\mathbb{C})$ is a $GL(n,\mathbb{C})$-valued stochastic process $G_t$
defined by the Stratonovich integral
$$ G_t = I + \int_0^t G_s \circ dW_s$$
where $W_t$ is Brownian motion in the vector space of $n \times n$ matrices with entries from $\mathbb{C}$.
(In other words, the Brownian motion on the Lie group $GL(n,\mathbb{C})$ is driven by the Brownian motion over its Lie algebra $\mathfrak{gl}(n,\mathbb{C})$.)
It has the property that for every $s \geq 0$, $(G_s^{-1}\cdot G_{t+s}, t \geq 0)$ has the same law of $(G_t, t \geq 0)$ and is independent of $(G_t, 0 \leq t \leq s)$.
A right-invariant Brownian motion over $GL(n, \mathbb{C})$ is the process $G_t^*$.

The process \begin{equation} \label{eqn:Yt} Y_t = G_t G_t^* \end{equation}
may be regarded as Brownian motion on the space of $n \times n$ Hermitian, positive-definite matrices \cite{OCJ}.
Indeed, one may identify the space $P(n)$ of $n \times n$ Hermitian positive-definite matrices as the symmetric
space $GL(n,\mathbb{C})/U(n)$ (essentially by the polar decomposition), and then Brownian motion on $GL(n,\mathbb{C})$ can be used to define Brownian motion on $P(n)$. See \cite{Bia, OCJ} for a discussion on this and more broadly of Brownian motion on symmetric spaces. We note further that $G_t G_t^*$ has the same eigenvalues as $G_t^* G_t$.

Let $\lambda_t^1 \geq \cdots \geq \lambda_t^n > 0$ be the eigenvalues of $Y_t$.
The log-transformed eigenvalues
$$ \gamma^i_t = \log( \lambda_t^i).$$
obey the system of SDEs (see \cite[Corollary 3.3]{OCJ})
$$ d\gamma^i_t = d\beta_t^i + \sum_{j: j\neq i} \mathrm{coth}(\gamma^i_t - \gamma^j_t)\, dt$$
where $\beta_t$ is Brownian motion on $\R^n$. It is shown in \cite[Proposition 4.2]{OCJ} that for
the largest eigenvalue at time $t=1$, $$\gamma^1_1 \stackrel{law}{=} \lambda_{max}\left (\mathrm{diag}(n-1,n-3,\ldots, -(n-3),-(n-1)) + H^{\mathrm{GUE}}\right).$$
In other words, in the notation of Theorem \ref{thm:arithmetic}, it has the law of the largest eigenvalue of \eqref{eqn:Hreg} with $\lambda_1^n = n-1$ and $\Delta = 2$.
Theorem \ref{thm:arithmetic} then implies
\begin{cor}
    In the limit as $n \to \infty$,
    $$ \mathbf{Pr}(\gamma^1_1 \leq n-1 + \frac{a + \log(n-1)}{2}) \to \dt{I - K_2}_{L^2[a,\infty)}.$$
\end{cor}

\subsection{Dyson's Brownian motion for GUE and universality of the Airy process} \label{sec:DBM}
Let $H(t)$ be Brownian motion in the space of $n \times n$ Hermitian matrices.
This can be expressed as
\begin{equation} \label{eqn:hbm}
H(t) = \frac{W(t) + W^*(t)}{\sqrt{2}}\end{equation}
where $W(t)$ is a matrix whose entries are i.i.d.~standard complex-valued Brownian motions.
So, $W_{i,j}(t) = \frac{W^1_{i,j}(t) + \mathbf{i}W^2_{i,j}(t)}{\sqrt{2}}$ where $W^{k}_{i,j}$ are independent, standard real-valued Brownian motions.

Let $\lambda_1(t) > \cdots > \lambda_n(t)$ denote the eigenvalues of $H(t)$. Dyson \cite{Dys} showed that these eigenvalues satisfy the system of SDEs
\begin{equation} \label{eqn:dysonsde}
    d \lambda_i(t) = d B_i(t) + \sum_{j: j \neq i} \frac{1}{\lambda_i(t)-\lambda_j(t)}\, dt
\end{equation}
where $B_1, \ldots, B_n$ are independent, standard real-valued Brownian motions.
This process of eigenvalues is known as Dyson's Brownian motion for GUE.

Let $W_1(t), \ldots, W_n(t)$ be independent, standard real-valued Brownian motions.
We can condition them to not intersect on $(0,\infty)$ by means of Doob's $h$-transform.
The harmonic function is the Vandermonde determinant
$$ h(x_1, \ldots, x_n) = \prod_{i< j} (x_i-x_j).$$
Upon conditioning, the process takes values in the Weyl chamber
$$ \mathbb{W} = \{ (x_1, \ldots, x_n) \in \R^n: x_1 > x_2 > \cdots > x_n\}.$$
It is known, \cite{Dys,Gra}, that these conditioned Brownian motions have the same
law as Dyson's Brownian motion for GUE \eqref{eqn:dysonsde}.

\subsubsection{The Airy process}
Consider the largest particle $\lambda_{max}(t) = \lambda_1(t)$ in Dyson's Brownian motion \eqref{eqn:dysonsde}.
Under the rescaling
$$ \bar{\lambda}_n(t) = \frac{\lambda_{max} \left( \frac{1 - 2tn^{-1/3}}{n}\right) - 2 + 2tn^{-1/3}}{n^{-2/3}}, \quad t \in \R,$$
the largest eigenvalue process converges to a limit process which is called the (parabolic) Airy process $\mathcal{A}(t)$.
The law of $\mathcal{A}(t)+t^2$ for every $t \in \R$ is the GUE Tracy-Widom distribution \cite{TW}.

In order to define $\mathcal{A}(t)$ we need to introduce the ``extended" Airy kernel \cite{FHN, PS}.
Let $\mathrm{Ai}(x)$ denote the Airy function:
$$ \mathrm{Ai}(x) = \frac{1}{\pi} \int_{0}^{\infty}dt\, \cos \left( \frac{t^3}{3} + xt\right).$$
The extended Airy kernel is an integral kernel on $\R^2$ defined by the formula
\begin{equation} \label{eqn:airykernel}
    K_{Airy}(t_1, x; t_2, y) = \begin{cases}
        \int_0^{\infty} dz\, e^{z(t_2-t_1)}\mathrm{Ai}(x+z) \mathrm{Ai}(y+z), & t_2 \leq t_1 \\
        - \int_{-\infty}^0dz\, e^{z(t_2-t_1)} \mathrm{Ai}(x+z) \mathrm{Ai}(y+z), & t_2 > t_1 
    \end{cases}.
\end{equation}
The extended Airy kernel defines a determinantal point process on $\R^2$ called the Airy line ensemble \cite{CH,PS}.
The Airy process is the top line of this ensemble. As such, its finite dimensional laws are given by Fredholm determinants
as follows. For $t_1 < t_2 < \cdots < t_m \in \R$ and $\xi_1, \ldots, \xi_m \in \R$,
\begin{equation} \label{eqn:airyfdds}
    \mathbf{Pr}(\mathcal{A}(t_i) \leq \xi_i, 1 \leq i \leq m) = \dt{I - K}_{L^2(\{1,\ldots,m\} \times [0,\infty))}
\end{equation}
where
$$ K(i,x;j,y) = K_{Airy}(t_i, x+\xi_i+t_i^2; t_j, y+\xi_j+t_j^2).$$
The finite dimensional laws form a consistent family and determine the law of $\mathcal{A}(t)$.

It is possible to rewrite the extended Airy kernel as a double contour integral.
Define the heat kernel
$$ e^{t \partial^2/2}(x,y) = \frac{1}{\sqrt{2 \pi t}} e^{- \frac{(x-y)^2}{2t}}$$
for $t > 0$. Define
\begin{equation} \label{eqn:Jairy}
    J_{Airy}(t_1, x; t_2, y) = \frac{1}{(2 \pi \mathbf{i})^2}\oint_{\Gamma_{-\delta}}dw\, \oint_{\Gamma_{\delta}}dz\, 
    \frac{e^{\frac{1}{3}z^3 + t_2 z^2 - yz}}{e^{\frac{1}{3}w^3 + t_1w^2-xw}} \cdot \frac{1}{z-w}
\end{equation}
where $\Gamma_d$ denotes the vertical contour $\{\Re(z) = d\}$ oriented upwards. The parameter $\delta > \max(|t_1|,|t_2|)$.
We have that, see \cite[Proposition 2.3]{Jo5},
\begin{equation} \label{eqn:Kixjy}
    K(i,x;j,y) = - e^{(t_j-t_i)\partial^2}(x+\xi_i, y+\xi_j)\ind{t_j > t_i} + J_{Airy}(t_i, x+\xi_i; t_j, y+\xi_j).
\end{equation}
There is some freedom in the choice of contours $\Gamma_{\pm \delta}$ as these may be deformed without changing the value of the integral.
For instance, $\Gamma_{\delta}$ can be deformed to the wedge-shaped contour $\{\delta_1+e^{\pm \mathbf{i}\frac{\pi}{3}}t, t \geq 0\}$ for any $\delta_1 \in \R$
and $\Gamma_{-\delta}$ can be deformed to $\{\delta_2 +e^{\pm \mathbf{i}\frac{\pi}{6}}, t \leq 0\}$ for any $\max\{|t_1|,|t_2|\} < \delta_2 < \delta_1$.

\subsubsection{Universality}
Let $H_0$ be an $n \times n$ Hermitian matrix.
Consider the process
$$X(t) = H(t) + H_0$$
and its largest eigenvalue $\lambda_{max}(X(t))$.
If $H_0$ has spectral decomposition $H_0 = U^* \Lambda U$ where $\Lambda = \mathrm{diag}(\lambda)$ is the diagonal matrix of eigenvalues and $U$ is unitary, then $U X(t) U^* = UH(t)U^* + \Lambda$ has the same eigenvalues as $X(t)$. Since $U H(t) U^*$ has the same law of $H(t)$, it follows that
$$ \lambda_{max}(X(t)) \stackrel{law}{=} \lambda_{max}(H(t) + \Lambda).$$
We expect that for suitably generic matrices $H_0$, $\lambda_{max}(X(t))$, properly rescaled, will converge to the Airy process.
More generally, the rescaled process of largest, 2nd largest, 3rd largest and so on eigenvalues of $X(t)$ should converge to the Airy line ensemble.

For a single time $t = \frac{\tau}{n}$, the asymptotic fluctuations of $\lambda_{max}(X(t))$ are indeed governed by the GUE Tracy-Widom distribution for suitable classes of matrices $H_0$ (both deterministic and random, most notably the Wigner matrices); see \cite{EY, Jo3, Shc, Sos, TV} and references therein.
In \cite[Theorem 1.2]{CNV} it is proven that $\lambda_{max}(X(t))$, appropriately rescaled, does converge to the Airy process under suitable assumptions on $H_0$.
These assumptions are formulated in terms of the limiting empirical distribution of the eigenvalues of $H_0$.
We prove the following Airy process limit theorem with assumptions expressed in terms of the eigenvalues of $H_0$ themselves.
Our assumptions should be close to those in \cite{CNV} as we use similar determinant formulas, although they don't appear to be the same.
In Proposition \ref{prop:Fab} we provide a sufficient criterion for our assumptions to hold in terms of the limiting empirical distribution of the eigenvalues of $H_0$.

Let $\nu = \{\nu_1, \ldots, \nu_n\} \subset \R$ be a collection of points counted with multiplicity (a point cloud).
For example, it could be the eigenvalues of an $n \times n$ Hermitian matrix.
Associate to $\nu$ the following constants. Let $b(\nu) > \max_j \nu_j$ be the unique real number such that
\begin{equation} \label{eqn:bnu}
    1 = \frac{1}{n} \sum_{j=1}^n \frac{1}{(b(\nu)-\nu_j)^2}.
\end{equation}
Let $a(\nu)$ be defined by
\begin{equation} \label{eqn:anu}
    a(\nu) = b(\nu) + \frac{1}{n}\sum_{j=1}^n \frac{1}{b(\nu)-\nu_j}.
\end{equation}
Let $d(\nu)$ be defined by
\begin{equation} \label{eqn:dnu}
    d(\nu) = \left (\frac{1}{n}\sum_{j=1}^n \frac{1}{(b(\nu)-\nu_j)^3} \right)^{1/3}.
\end{equation}

Suppose $0 < \alpha \leq \beta < \infty$. Define the set of point clouds
$$ F(\alpha,\beta) = \{ \nu = \{\nu_1, \ldots, \nu_n\}\;\text{for some}\; n: \alpha \leq b(\nu)-\nu_j  \leq \beta\;\text{for}\; 1\leq j \leq n\}.$$
Note that if $\nu \in F(\alpha, \beta)$ then $1/\beta \leq d(\nu) \leq 1/\alpha$.

Let $H^0_n$ be a sequence of $n \times n$ Hermitian matrices with eigenvalues $\nu^n = \{\nu^n_1, \ldots, \nu^n_n\}$.
Denote $H_n(t)$ to be Brownian motion in the space of $n \times n$ Hermitian matrices (see \eqref{eqn:hbm}).
Let $b_n = b(\nu^n)$, $a_n = a(\nu^n)$ and $d_n = d(\nu^n)$.
Consider the process
\begin{equation} \label{eqn:lambdant}
    \lambda_n(t) = \lambda_{max}\left (H_n \left(\frac{1-2d_n^2tn^{-1/3}}{n}\right) + H^0_n \right ), \quad t \leq \frac{n^{1/3}}{2d_n^2}.
\end{equation}
 Assume there exists $0 < \alpha \leq \beta < \infty$
such that $\nu^n \in F(\alpha,\beta)$ for all sufficiently large $n$.
\begin{thm} \label{thm:airyconvergence}
    Under the assumptions above, as $n \to \infty$, we have convergence in law of the process
    $$ \bar{\lambda}_n(t) = \frac{\lambda_n(t) - a_n - 2t d_n^2(b_n-a_n)n^{-1/3}}{d_n n^{-2/3}} \to \mathcal{A}(t).$$
\end{thm}

Proposition \ref{prop:Fab} gives a criterion to check when a sequence of point clouds belong to $F(\alpha,\beta)$ for suitable $\alpha$ and $\beta$.

\subsection{Noncolliding Brownian bridges and point-to-line last passage percolation} \label{sec:flat}

Consider the following model. Let $b: [0,\infty) \to \mathbb{R}$ be a continuous function with $b(0) = 0$. Let $\mu_1,\mu_2,\ldots$ be a sequence of real numbers.
Let $(B^{\mu}_k(t): t \geq 0)$ for $k = 1, 2, 3, \ldots$ be a collection of independent Brownian motions such that $B^{\mu}_k$ has drift $\mu_k$:
$$ B^{\mu}_k(t) = B^{\mathrm{st}}_k(t) + \mu_k t$$
where $B^{\mathrm{st}}_k$ is a standard real-valued Brownian motion.

The Brownian last passage percolation (BLPP) with boundary $b$ is a process $BLPP(b;(t,m))$ for $t \geq 0$ and $m \geq 1$
defined as follows:
\begin{equation} \label{eqn:BLPP}
BLPP(b;(t,m)) = \max_{0 \leq t_0 \leq t_1 \leq \cdots \leq t_m = t} b(t_0) + \sum_{k=1}^m B^{\mu}_k(t_k) - B^{\mu}_k(t_{k-1}).
\end{equation}

Classical BLPP considers the so-called narrow wedge boundary condition whereby $b(0)=0$ and $b(t) = -\infty$ otherwise. In this case,
$$BLPP((0,0);(t,m)) = \max_{0 = t_0 \leq t_1 \leq \cdots \leq t_m = t} \sum_{k=1}^m B^{\mu}_k(t_k) - B^{\mu}_k(t_{k-1}).$$
Of course the narrow wedge in not a continuous function, but it can be approximated with the functions $b_{L}(t) = - Lt$ in the limit
as $L \to \infty$, upon which it is easy to check that $BLPP(b_L; (t,m))$ converges to $BLPP((0,0);(t,m))$ almost surely.

Consider the narrow wedge boundary with all drifts equal to zero. Then the random variable $BLPP((0,1);(1,m))$ has the same law as the largest eigenvalue of an $m\times m$ GUE random matrix \cite{GTW}. More generally, the process $m \mapsto BLPP((0,1);(1,m))$ has the law of the largest eigenvalue of the minors of an infinite GUE random matrix \cite{Bar}.

The process $t \mapsto BLPP((0,0);(t,m))$ has an interpretation in terms of noncolliding Brownian motions.
Suppose $(\hat{B}_1 > \hat{B}_2 > \cdots > \hat{B}_m)$ are $m$ independent Brownian motions conditioned not to intersect in the sense of Doob's $h$-transform with harmonic function $h(x_1, \ldots, x_n) = \prod_{i < j}\, (x_i-x_j)$ on the domain $\mathbb{W} = \{(x_1,x_2,\ldots, x_n) \in \R^n: x_1 > x_2 > \ldots > x_n\}$. Then $$\hat{B}_1(t) \stackrel{d}{=} BLPP((0,0);(t,m))$$ as a process in $t$ \cite{OY}. Furthermore, $\hat{B}_1$ has the law of the trajectory of the top particle among $m$ particles performing Dyson's Brownian motion for GUE \cite{Gra}.

Another boundary condition of interest is the flat boundary $b(t) \equiv 0$. In this case,
$$BLPP(0;(t,m)) = \max_{0 \leq t_0 \leq t_1 \leq \cdots \leq t_m = t} \sum_{k=1}^m B^{\mu}_k(t_k) - B^{\mu}_k(t_{k-1}).$$
Fitzgerald and Warren \cite{FW} (see also \cite{BFPSW}) have studied this case, drawing a connection to random matrices
and a point-to-line last passage problem.

Let $\mu \in \R^n$ and consider the $n \times n$ matrix $$ M(t) = H(t) + t \cdot  \mathrm{diag}(\mu).$$ Let $\lambda_{max}(t)$ be the largest eigenvalue of $M(t)$.
We are interested in the process
$$ Z_n(t) = \sup_{0 \leq s \leq t} \lambda_{\max}(s) = \sup_{0 \leq s \leq t}\, \lambda_{max}\left (H(s) + s \cdot \mathrm{diag}(\mu) \right).$$
Fitzgerald and Warren \cite[Proposition 4]{FW} prove that for every fixed $t$,
\begin{equation} \label{eqn:eigenidentity} Z_n(t) \stackrel{law}{=} BLPP(0; (t,n)).\end{equation}
In Proposition \ref{prop:flat} we give a Fredholm determinant formula for the law of $Z_n(t)$.

Now suppose all drifts in $\mu$ are negative: $\mu_i = - \beta_i$ with $\beta_i > 0$ for every $i$. Fitzgerald and Warren \cite[Theorem 1]{FW} prove that
\begin{equation} \label{eqn:fw}
\lim_{t \to \infty} Z(t) = \sup_{t \geq 0} \; \lambda_{max}\left( H(t) - t \cdot \mathrm{diag}(\beta) \right) \stackrel{law}{=} \Pi_{flat}
\end{equation}
where $\Pi_{flat}$ is the following random variable obtained via a last passage percolation problem.

Consider the set $\Delta = \{(i,j): i,j \geq 1\; \text{and}\; i+j \leq n+1\}$ and it boundary $\partial \Delta = \{(i,j) \in \Delta: i+j=n+1\}$.
An up/right path in $\Delta$ is a lattice path from $(1,1)$ to $\partial \Delta$ such that each step of the path goes in the direction $(1,0)$ of $(0,1)$.
For example, $(1,1) \to (2,1) \to (2,2) \to (3,2)$ is an up/right path. Put on the point $(i,j) \in \Delta$ a random variable
$\omega_{i,j} \sim \mathrm{Exponential}(\beta_i + \beta_{n+1-j})$ such that they are all independent. Then,
\begin{equation}
    \Pi_{flat} = \max_{\text{all up/right paths} \;\pi}\; \sum_{(i,j) \in \pi} \omega_{i,j}.
\end{equation}
Note that when $n=1$, $\Pi_{flat} \stackrel{law}{=} \mathrm{Exponential}(2 \beta_1)$, which recovers the well-known identity
$$ \sup_{t \geq 0} \; B(t) - \beta t \stackrel{law}{=} \mathrm{Exponential}(2 \beta)$$
where $B$ is a standard Brownian motion and $\beta > 0$.

We have the following determinant formula for the c.d.f.~of $\Pi_{flat}$.
\begin{thm} \label{thm:Piflat}
    For $a \in \R$,
    $$ \mathbf{Pr}(\Pi_{flat} \leq a) = \det(I - \chi K \chi)_{L^2(\R)}$$
    where $\chi(x) = \ind{x \geq \max\{a,0\}}$ and
    $$K(x,y) = - \frac{1}{2 \pi \mathbf{i}} \oint_{\gamma}dw\, e^{-(x+y)w} \prod_{i=1}^n \frac{\beta_i+w}{\beta_i-w}$$
    where $\gamma$ is a counter clockwise oriented closed contour containing all the poles at $w = \beta_i$.
\end{thm}

O'Connell \cite[Section 10.2]{OC2} also gives a formula for $\sup_{t \geq 0} \; \lambda_{max}\left( H(t) - t \cdot \mathrm{diag}(\beta) \right)$ as a ratio of determinants.
The formula reads, for $a \geq 0$,
\begin{equation} \label{eqn:detratio}
\mathbf{Pr}\left( \sup_{t \geq 0} \; \lambda_{max}\left( H(t) - t \cdot \mathrm{diag}(\beta) \right) \leq a \right) = \frac{\det(\beta_i^{j-1} - e^{-2\beta_i a}(-\beta_i)^{j-1})_{i,j}}
{\det(\beta_i^{j-1})_{i,j}}.
\end{equation}
Actually, the above formula comes from taking a limit of the formula for $f(\nu,y)$ in \cite{OC2} but we omit the standard details. Theorem \ref{thm:Piflat} then implies a Fredholm determinant formula for this ratio of determinants.

\subsubsection{Noncolliding Brownian bridges} \label{sec:bridges}

The process $Z_n(t)$ is closely related to noncolliding Brownian bridges.
Let $$(B^{br}_1(t), \ldots, B^{br}_n(t))$$ with $B^{br}_1 > B^{br}_2 > \cdots > B^{br}_n$ be $n$ Brownian bridges
starting from zero at time $0$ and ending at zero at time $1$ and conditioned not to intersect on $(0,1)$. There are a few ways to do this conditioning.
One way is to realise it as the Doob $h$-transform of $n$ independent bridges with a suitable harmonic function $h$ \cite{Gra}. Another way is to consider
condition the independent bridges to not intersect on the time interval $(\epsilon, 1-\epsilon)$ and then take the limit $\epsilon \to 0$ \cite{CH}.
The third way, which if most relevant to our discussion, is to consider $H^{br}(t) = (1-t) H(\frac{t}{1-t})$ for $t \in [0,1]$, which is a bridge in the space of Hermitian matrices.
The ordered eigenvalues of $H^{br}(t)$ have the same law as $(B^{br}_1(t), \ldots, B^{br}_n(t))$, an observation that essentially dates back to Dyson \cite{Dys} (see also \cite{Gra}).

We shall see that
\begin{equation} \label{eqn:bridgeeigenvalue}
\left( \max_{t\in[0,1]} B^{br}_1(t) \right)^2 \stackrel{law}{=} \sup_{t \geq 0} \lambda_{max}(H(t)-tI).
\end{equation}
Nguyen and Remenik \cite[Theorem 1.2]{NR} (see also \cite{RS,SMCR}) have shown that
\begin{equation} \label{eqn:NR}
\left( \max_{t\in[0,1]} B^{br}_1(t) \right)^2 \stackrel{law}{=} \frac{1}{4} \lambda_{max}(X^tX)
\end{equation}
where $X$ is an $(n+1) \times n$ matrix with i.i.d. standard Normal entries (real valued).
The matrix $X^tX$ is known as a certain case of the Laguerre Orthogonal Ensemble\footnote{In the general case, $X$ is an $m \times n$ matrix with $m \geq n$ and the joint eigenvalue density $f$ acquires a factor of $\prod_i \lambda_i^a$ with $a = (m-n-1)/2$.},
and the joint law of its eigenvalues is given by
$$ f(\lambda_1,\ldots, \lambda_n) = \frac{1}{Z_n} \prod_{1 \leq i<j\leq n} |\lambda_j-\lambda_i| \cdot e^{-\frac{1}{2} \sum_{i=1}^n \lambda_i},\quad \lambda_i \geq 0.$$
Thus, we have the following corollary of \eqref{eqn:fw} and Theorem \ref{thm:Piflat}.

\begin{cor} \label{cor:loe}
    Let $X$ be a $(n+1) \times n$ random matrix whose entries are i.i.d. standard Normal random variables. The largest eigenvalue $\lambda_{max}(X^tX)$
    of the Laguerre Orthogonal Ensemble has c.d.f.~given by
    $$ \mathbf{Pr}(\lambda_{max}(X^tX) \leq 4a) = \det(I -\chi K \chi)_{L^2(\R)}$$
    where $\chi(x) = \ind{x \geq \max \{a,0\}}$ and
    $$ K(x,y) = - \frac{1}{2 \pi \mathbf{i}} \oint_{\gamma}dw\, e^{-(x+y)w} \left (\frac{1+w}{1-w} \right)^n$$
    with $\gamma$ being a counter clockwise oriented closed contour containing 1.
\end{cor}

Now consider $n$ independent Brownian bridges $(B^{br}_1, \ldots, B^{br}_n)$ such that $B^{br}_i(t)$ starts at $\nu_i$ at $t=0$ and ends at zero at $t=1$.
Condition these bridges to not collide on $(0,1)$ and denote by $B^{br}_{top}(t)$ the topmost bridge after conditioning. For $r > \max \{\nu_1, \ldots, \nu_n, 0\}$, let
\begin{equation} \label{eqn:Fbridge} F(r) = \mathbf{Pr}\left (\max_{t \in [0,1]} \,B^{br}_{top}(t) \leq r\right ).\end{equation}
It is explained in \cite[section 10.2]{OC2} how $F(r)$ is related to the right-hand-side of \eqref{eqn:detratio}; it equals that ratio for $\beta_i = r-\nu_i$ and $a=r$.
So as a corollary to Theorem \ref{thm:Piflat} we have
\begin{cor} \label{cor:bridgemax}
    The c.d.f.~$F(r)$ from \eqref{eqn:Fbridge} equals
    $$ F(r) = \det(I - K)_{L^2[0,\infty)}$$
    where
    $$K(x,y) = - \frac{1}{2\pi \mathbf{i}} \oint_{\gamma} dw\, e^{-(x+y+2r^2)w} \prod_{i=1}^n \frac{1+w - (\nu_i/r)}{1-w-(\nu_i/r)}$$
    and $\gamma$ is a closed, counter clockwise oriented contour enclosing all the poles at $1-(\nu_i/r)$.
\end{cor}

Liechty, Nguyen and Remenik \cite{LNR} consider a certain KPZ scaling limit of $\max_{t} B^{br}_{top}(t)$ above.
Under this scaling, $B^{br}_{top}$ converges to a process called the Airy process with wanderers \cite{AFM} -- it is a deformation of the Airy process introduced before.
They find the limiting c.d.f.~$F_{limit}$ of $F$ under the aforementioned KPZ scaling, see \cite[Theorem 1]{LNR}, and describe it in many ways.
It is the c.d.f.~of the all time maximum of the (parabolic) Airy process with wanderers. The maximum of the Airy process is the GOE Tracy-Widom distribution \cite{Jo5}.
It is argued in \cite{LNR} that $F_{limit}$ is a deformation of the GOE Tracy-Widom distribution. One can thus think of Corollary \ref{cor:bridgemax} as a finite $n$ analogue
of this result.

The article \cite{LNR} also finds intriguing connections between the formula for $F_{limit}$ and solutions to differential equations such as the Painlev\'e II equation,
the KdV equation and a p.d.e.~introduced by Bloemendal and Vir\'ag \cite{BV1, BV2} to describe the largest eigenvalue of ``spiked" random matrices.
It may be interesting to understand the finite $n$ analogue of these connections.

\section*{Acknowledgements}
I thank Neil O'Connell for bring my attention to \cite{OC2} and Tom Claeys for telling me about \cite{CNV} as well as related references in the literature.

\section{Preliminaries} \label{sec:prem}
Let $K$ be an integral kernel acting on the space $L^2(\Omega,\mu)$. The Fredholm determinant of $K$ is
\begin{equation} \label{eqn:Fredholm}
\dt{I-K} = 1 + \sum_{k=1}^{\infty} \frac{(-1)^k}{k!} \int_{\Omega^K} d\mu(z_1) \cdots d\mu(z_k)\, \dt{K(z_i,z_j)}_{1\leq i,j\leq k.}\end{equation}
If there are functions $f$ and $g$ on $X$ such that
$$ |K(x,y)| \leq f(x) g(y)$$
with $f$ bounded and $g$ integrable (or vice versa), then the series converges absolutely.
Furthermore, suppose a sequence of integral kernels
$K_n$ satisfy $K_n \to K$ pointwise on $\Omega$ and $|K_n(x,y)| \leq f(x) g(y)$ for every $n$ with $f$ bounded and $g$ integrable (or vice versa). Then,
$$ \dt{I+K_n}_{L^2(\Omega,\mu)} \to \dt{I+K}_{L^2(\Omega,\mu)}.$$
See \cite{Josurvey} for proofs of these facts, which are deduced from Hadamard's inequality and the dominated convergence theorem.

We will be interested in Fredholm determinants over the space $\Omega = \{1,2,\ldots,m\} \times \R$ with the measure $\mu$ being the product of counting measure on $\{1,2,\ldots,m\}$ and Lebesgue measure on $\R$.

A conjugation of an integral kernel $K$ is a kernel $K'$ of the form
$$ K'(x,y) = \frac{c(x)}{c(y)} K(x,y)$$
where $c: \Omega \to \R\setminus{0}$ is a non-vanishing function.
Fredholm determinants remain invariant under conjugation: $\dt{I-K'} = \dt{I-K}$.
Although, bounds of the form $|K(x,y)| \leq f(x)g(y)$ are \emph{not} invariant under conjugation.
In order to demonstrate that the series expansion of a Fredholm determinant is absolutely convergent, it is often necessary to conjugate the kernel $K$ so that the conjugated kernel $K'$ does satisfy a bound of the form $K'(x,y)| \leq f(x)g(y)$ with $f$ bounded and $g$ integrable (or vice versa).
It is customary, for sake of simplicity, to not include such conjugation factors when presenting a kernel in a theorem, although in proofs we do conjugate kernels so that the Fredholm series expansion converges absolutely.

\section{Brownian last passage percolation with drifts and boundary} \label{sec:blpp}

In this section we provide a Fredholm determinant formula for the model \eqref{eqn:BLPP} of Brownian last passage percolation with drifts and a boundary.

For $t > 0$, recall the heat kernel
$$e^{t \partial^2/2}(x,y) = \frac{1}{\sqrt{2\pi t}} e^{-(x-y)^2/2t}.$$

For $m \geq 1$ and $t > 0$, we define the following family of integral kernels.
\begin{equation} \label{eqn:Smt}
    S_{m,-t}(x,y) = \frac{1}{2\pi \mathbf{i}} \oint_{\gamma} dz\, e^{-\frac{t}{2}z^2 + (x-y)z} \prod_{i=1}^m (z-\mu_i)^{-1}.
\end{equation}
Here $\gamma$ is a closed contour that contains all the poles at $z = \mu_1, \mu_2, \ldots, \mu_m$ and is oriented counter clockwise.
\begin{equation} \label{eqn:Sbarmt}
    \bar{S}_{m,t}(x,y) = \frac{1}{2\pi \mathbf{i}} \oint_{\Gamma} dz\, e^{\frac{t}{2}z^2 + (x-y)z} \prod_{i=1}^m (z-\mu_i).
\end{equation}
Here $\Gamma$ is a vertical contour $\{\Re(z) = d\}$ for any $d \in \mathbb{R}$ and is oriented upwards.

Let $W$ be a standard Brownian motion started from $x \in \mathbb{R}$ and define, for the boundary condition $b$, the stopping time
$$ \tau = \inf \{ s \geq 0: W(s) \leq b(s)\}.$$
Define the integral kernel
\begin{equation} \label{eqn:Shypo}
    S^{\mathrm{hypo}(b)}_{m,t}(x,y) = \mathbf{E} \left [ \bar{S}_{m,t-\tau}(W_{\tau},y) \ind{\tau \leq t} \mid W(0) = x \right ].
\end{equation}

Let $D_{m,y}$ be the operator
\begin{equation}
    D_{m,y} = \prod_{i=1}^m (-\partial_y - \mu_i).
\end{equation}
Observe that
$\bar{S}_{m,t}(x,y) = D_{m,y} \cdot e^{t \partial^2/2}(x,y).$
As such, we can write
$$ S_{m,t}^{\mathrm{hypo}(b)}(x,y) = D_{m,y} \cdot S_t^{\mathrm{hit}(b)}(x,y)$$
where
$$S_t^{\mathrm{hit}(b)}(x,y) = \mathbf{E} \left[ e^{(t-\tau) \partial^2/2}(W(\tau),y) \ind{\tau \leq t} \mid W(0)=x\right] = \mathbf{Pr}(\tau \leq t, W(t) \in dy \mid W(0)=x).$$

Finally, define the extended integral kernel $K(t_1,x; t_2,y)$ according to
\begin{equation}\label{eqn:KBrownian}
K(t_1, x; t_2, y) = - e^{(t_2-t_1)\partial^2/2}(x,y) \mathbf{1}_{t_1 < t_2} + S_{m, -t_1} \cdot S^{\mathrm{hypo}(b)}_{m,t_2}(x,y).
\end{equation}
The kernel acts on the space $L^2(\{t_1,\ldots, t_k\} \times \mathbb{R})$.

For $(t_1,a_1), \ldots, (t_k,a_k) \in \mathbb{R}^2$, define the operator $\chi_a$ according to $\chi_a(t_i,x) = \mathbf{1}_{x \geq a_i}$.
It also acts on $L^2(\{t_1,\ldots, t_k\} \times \mathbb{R})$.

\begin{thm} \label{thm:BLPP}
For $m \geq 1$, $0 < t_1 < t_2 < \cdots < t_k$ and $a_1, \ldots, a_k \in \mathbb{R}$,
$$\mathbf{Pr}(BLPP(b; (t_i,m)) \leq a_i; 1 \leq i \leq k) = \det(I - \chi_a K \chi_a)_{L^2(\{t_1,\ldots, t_k\} \times \mathbb{R})}.$$
\end{thm}

\begin{rem}
    Note that the kernel $K$ in Theorem \ref{thm:BLPP} is left invariant under permutations of the drifts $\mu_i$.
    So the law of $BLPP(b; (t,m))$ is invariant under permutations of the drifts, which is known as the Burke property.
\end{rem}

In order to prove this theorem, we will consider a model of inhomogeneous Geometric last passage percolation and take an appropriate scaling limit to the Brownian model. Such limit transitions were first obtained by Glynn and Whitt \cite{GW}.

\subsection{Inhomogeneous Geometric last passage percolation} \label{sec:geomlpp}
Let $a_i \in (0,1)$ be a sequence of numbers for $i \geq 1$.
Let $\omega_{i,j}$ be independent random variables with law
$$ \mathbf{Pr}(\omega_{i,j} = k) = (1-a_i) a_i^{k} \quad k = 0,1,2, \ldots$$
Let $G(m,n)$ be the the last passage time from $(0,0)$ to $(m,n)$, defined recursively by
$$G(m,n) = \max \{G(m-1,n), G(m,n-1)\} + \omega_{m,n}.$$
The initial condition is on column zero ($m = 0$) by setting $G(0,n) = x_n$ with integers $0 \leq x_1 \leq x_2 \leq x_3 \leq \cdots$,
and $G(m,0) = 0$ for $m \geq 0$.

Let $N$ be large and fixed integer. Define the random vector
$$ G(m) = (G(m,1), G(m,2), \ldots, G(m,N))$$
It takes values in the set
$$ \mathbb{W}_N = \{(x_1, \ldots, x_N) \in \mathbb{Z}^N: x_i \leq x_{i+1}\}$$
In \cite[Theorem 1]{JR} the following formula for the transition probability of $G(m)$ is proved (note that it is an inhomogeneous Markov process):
\begin{equation} \label{eqn:transition}
P(G(m) = y \mid G(0)=x) = \det [W_{j-i}(y_j-x_i)]_{1\leq i,j \leq N}
\end{equation}
with
$$ W_k(x) = \frac{1}{2 \pi \mathbf{i}} \oint_{|z| = r > 1} dz\, z^{x-1} (z-1)^k \prod_{i=1}^m \frac{1-a_i}{1- a_i/z}.$$

We define the following bijection. For $x = (x_1, \ldots, x_N) \in \mathbb{W}_N$ define
$$ \tilde{x}_j = -x_j - j. $$
We have that $\tilde{x} \in \tilde{\mathbb{W}}_N$ where
$$ \tilde{\mathbb{W}}_N = \{ x = (x_1, \ldots, x_N) \in \mathbb{Z}^N: x_1 > x_2 > \cdots > x_N\}.$$
The mapping $x \mapsto \tilde{x}$ is a bijection from $\mathbb{W}_n$ to $\tilde{\mathbb{W}}_N$.
Define
$$ \tilde{W}_k(x) = W_{-k}(k-x).$$
Then
\begin{equation}
    \tilde{W}_k(x) = \frac{(-1)^k}{2\pi \mathbf{i}} \oint_{|z|=r>1} \frac{dz}{z}\, \frac{z^{k-x}}{(1-w)^k}\, \phi_m(z)
\end{equation}
where
$$ \phi_m(z) = \prod_{i=1}^m \frac{1-a_i}{1- a_i/z}.$$

\begin{prop} \label{prop:schutz}
    For $x,y \in \mathbb{W}_N$, we have
    $$P(G(m)=y \mid G(0) = x) = \det [\tilde{W}_{i-j}(\tilde{y}_{N+1-i}-\tilde{x}_{N+1-j})]$$
\end{prop}

\begin{proof}
Define $w_k(x) = W_{-k}(-x)$. Using \eqref{eqn:transition} we find that
\begin{align*}
  P(G(m) = y \mid G(0)=x) &= \det (W_{j-i}(y_j-x_i)) \\
  &= \det [w_{i-j}(x_i-y_j] \\
  &= \det [w_{j-i}(x_{N+1-i}-y_{N+1-j})]\\
  &= \det [w_{i-j}(x_{N+1-j}-y_{N+1-i})]\\
  &= \det [w_{i-j}(-\tilde{x}_{N+1-j} - (N+1-j) +\tilde{y}_{N+1-i} + (N+1-i)]\\
  &= \det [w_{i-j}(\tilde{y}_{N+1-i} - \tilde{x}_{N+1-j} +j-i)] \\
  &= \det [\tilde{W}_{i-j}(\tilde{y}_{N+1-i} - \tilde{x}_{N+1-j})].
\end{align*}
\end{proof}

Define a new process $X(m) \in \tilde{W}_N$ according to
$$X(m,n) = -G(m,n)-n$$
By Proposition \ref{prop:schutz}, for $x,y \in \tilde{W}_N$,
\begin{equation} \label{eqn:Xtrans}
P(X(m) = y \mid X(0) = x) = \det [\tilde{W}_{i-j}(y_{N+1-i}-x_{N+1-j})]
\end{equation}
Also, we have that
\begin{equation} \label{eqn:Xfdd}
    \mathbf{Pr}(G(m,n_i) +n_i < a_i; \,1 \leq i \leq k \mid G(0)=x) = \mathbf{Pr}(X(m,n_i) > -a_i; \, 1 \leq i \leq k \mid X(0)=\tilde{x}).
\end{equation}

By \cite[Theorem 1.2]{MR} we have, due to \eqref{eqn:Xtrans}, that
\begin{equation} \label{eqn:Xdeterminant}
   \mathbf{Pr}(X(m,n_i) > -a_i; \, 1 \leq i \leq k \mid X(0)=\tilde{x}) = \det(I - \bar{\chi}_a K_G \bar{\chi}_a)_{\ell^2(\{n_1, \cdots, n_m\}\times \mathbb{Z})}
\end{equation}
where $\bar{\chi}_a(i,z) = \ind{z \leq -a_i}$ and the kernel $K_G$ is as follows.

\begin{equation} \label{eqn:Kgeom}
    K_G(n_1,x; n_2, y) = - Q^{n_2-n_1}\mathbf{1}_{n_1 < n_2}(x,y) + S_{m,-n_1} \cdot S^{\mathrm{epi}(\tilde{x})}_{m,n_2}(x,y)
\end{equation}
The kernels $Q$, $S$ and $S^{hypo(x)}$ are as follows for integers $m,n \geq 1$.

Choose $\theta \in (0,1)$ and set $\alpha = (1-\theta)/\theta$.
Let $\max_{1 \leq i \leq m} a_i < r <1$ and $\delta < 1$ be radii parameters.

\begin{equation} \label{eqn:Qgeom}
Q^n(z_1,z_2) = \frac{1}{2\pi \mathbf{i}} \oint_{|w|=r} dw\, \frac{\theta^{z_1-z_2}}{w^{z_2-z_1-n+1}} \left(\frac{\alpha}{1-w}\right)^n.
\end{equation}
This is the $n$-step transition probability of a random walk $B_n$ with $- Geom(1-\theta)$ steps [strictly to the left].
\begin{equation} \label{eqn:Sgeom}
S_{m,-n}(z_1,z_2) = \frac{1}{2 \pi \mathbf{i}} \oint_{|w|=r}dw\, \frac{\theta^{z_1-z_2}}{w^{z_1-z_2+n+1}} \left( \frac{1-w}{\alpha}\right)^n \phi_m(w).
\end{equation}
\begin{equation} \label{eqn:Sbargeom}
\bar{S}_{m,n}(z_1,z_2) = \frac{1}{2 \pi \mathbf{i}} \oint_{|w|=\delta}dw\, \frac{\theta^{z_1-z_2} (1-w)^{z_2-z_1+n-1}}{(w/\alpha)^n} \phi_m(1-w)^{-1}.
\end{equation}
Let
$$\tau = \min \, \{n=0,1,2, \ldots: B_n > \tilde{x}_{n+1}\}.$$
The kernel
\begin{equation} \label{eqn:Shypogeom}
S^{\mathrm{epi}(\tilde x)}_{m,n}(z_1,z_2) = \E{ \bar{S}_{m,n-\tau}(B_{\tau},z_2) \mathbf{1}_{\tau < n} \mid B_0 = z_1}.
\end{equation}

See \cite[Section 3]{Rah} for more details behind this derivation of the kernel from \cite{MR}.

\subsection{Limit transition from Geometric to Brownian}
We explain the limit transition from Geometric LPP to Brownian LPP in our setting.
Define constants
$$ c_1 = 1  \quad c_2 = 2 \quad c_3 = \frac{1}{2 \sqrt{2}}.$$
In the Geometric model, we choose
$$ a_i = \frac{1}{2} + \frac{c_3 \mu_i}{\sqrt{N}} \quad \text{for}\; N \;\text{a large integer}.$$
Let $c_{1,i} = \mathbf{E}[w_{i,j}] = \frac{a_i}{1-a_i}$ and $c_{2,i} = \mathrm{Var}(w_{i,j}) = \frac{a_i}{(1-a_i)^2}$.
A computation shows that
$$ c_{1,i} = c_1 + \frac{4c_3\mu_i}{\sqrt{N}} + O(1/N) \quad c_{2,i} = c_2 + \frac{12c_3\mu_i}{\sqrt{N}} + O(1/N).$$

For $k \geq 1$ define the random walk
$$ S^{(k)}(j) = w_{k,1} + w_{k,2} + \cdots + w_{k,j}; \quad S^{(k)}_j = 0 \;\text{for}\; j \leq 0.$$
Extend it to $j \in \mathbb{R}$ by linear interpolation.
Finally, define
$$ X^{(k}(t) = \frac{S^{(k)}(Nt) - c_{1,k}Nt}{\sqrt{c_{2,k} N}}.$$
By Donsker's Theorem, there is a coupling of the walks $X^{(k)}$ with i.i.d. standard Brownian motions $B^{\mathrm{st}}_k$, $k=1,2,\cdots$, such that
$$ (X^{(1)}, X^{(2)}, \ldots ) \to (B^{\mathrm{st}}_1,B^{\mathrm{st}}_2, \ldots) $$
uniformly on compacts, almost surely.
A computation shows that
$$X^{(k)}(t) = \frac{S^{(k)}(Nt)- c_1Nt}{\sqrt{c_2 N}} - 2\sqrt{2}c_3 \mu_k t + O(N^{-1/2}).$$
It follows that
$$ \frac{S^{(k)}(Nt)- c_1Nt}{\sqrt{c_2 N}} \to B_k^{\mathrm{st}}(t) + \mu_k t =: B^{\mu}_k(t)$$
for every $k$, where the convergence holds uniformly on compacts, almost surely.

We have that
$$G(m,n) = \max_{0 \leq n_0 \leq n_1 \leq \cdots \leq n_m=n} x_{n_0} + (S^{(1)}(n_1) - S^{(1)}(n_0-1) + \cdots + (S^{(m)}(n_m)-S^{(m)}(n_{m-1}-1))$$
We thus find that
$$ \frac{G(m,n) - c_1(n+m)}{\sqrt{c_2 N}} = \max_{0 \leq n_0 \leq n_1 \leq \cdots \leq n_m=n}
\frac{x_{n_0} - c_1n_0}{\sqrt{c_2 N}} + \sum_{k=1}^m X^{(k)}(n_k) - X^{(k)}(n_{k-1}-1)$$

Let us assume that
$$\frac{x_{\lfloor Nt \rfloor} - c_1Nt}{\sqrt{c_2 N}} \to b(t) \quad \text{uniformly on compacts}.$$
This will be the case if we start with a continuous $b(t)$ and then define $$x_n = c_1n + \lfloor \sqrt{c_2N}b(n/N)\rfloor.$$

Then we find that
$$ \frac{G(m, \lfloor Nt \rfloor) - c_1(Nt+m)}{\sqrt{c_2 N}} \to BLPP(b; (t,m))$$
where
$$BLPP(b;(t,m)) = \max_{0 \leq t_0 \leq t_1 \leq \cdots \leq t_m = t}
b(t_0) + \sum_{k=1}^m B^{\mu}_k(t_k) - B^{\mu}_k(t_{k-1})$$

To get the finite dimensional laws of $BLPP(b; (\cdot,m))$ one has to compute the large $N$ limit of
$$ \mathbf{Pr}(G(m,Nt_i) < c_1Nt_i + a_i\sqrt{c_2 N}; 1 \leq i \leq k).$$
This is equal to
$$ \mathbf{Pr}(X(m,Nt_i) > -(c_1+1)Nt_i - a_i \sqrt{c_2 N}; 1 \leq i \leq k)$$
for which we have the formula \eqref{eqn:Xdeterminant}.

\subsection{Proof of Theorem \ref{thm:BLPP}}

In order to prove the theorem, we have to analyse the asymptotics of
$$Pr(X(m,Nt_i) > -2Nt_i - a_i \sqrt{2N}; 1 \leq i \leq k)$$
using the right hand side of \eqref{eqn:Xdeterminant}. We have the extended kernel from \eqref{eqn:Kgeom};
after rescaling the Fredholm determinant, we need to analyse the asymptotics of the rescaled kernel
$$\sqrt{2N} K_G(Ns, -2Ns - x\sqrt{2N}; Nt, -2Nt-y\sqrt{2N})$$
with $s,t > 0$ and $x,y \in \mathbb{R}$. We need to show that it converges in the appropriate sense to the extended kernel
$$K(s,x;t,y)$$ in the statement of Theorem \ref{thm:BLPP} so that the Fredholm determinants also converge in the large $N$ limit.

Firstly we shall derive the limits of the constituent kernels $S_{m,-n}$, $\bar{S}_{m,n}$ and $S^{\mathrm{hypo}(x)}_{m,n}$
and then establish some decay estimates to get the limit of the Fredholm determinants.

Throughout the proof we shall choose the parameter $\theta = 1/2$; so $\alpha = 1$.

\subsubsection{Limits of the constituent kernels}
\begin{lem} \label{lem:Qlimit}
    Suppose $s < t$. Let $n_1 = Ns$, $n_2 = Nt$, $z_1 = -2Ns - x\sqrt{2N}$ and $z_2 = -2Nt -y\sqrt{2N}$. Then
    $$ \sqrt{2N} Q^{n_2-n_1}(z_1,z_2) \to e^{(t-s)\partial^2/2}(x,y).$$
    pointwise in $x,y \in \R$. Furthermore, one has the decay estimate: $ |\sqrt{2N} Q^{n_2-n_1}(z_1,z_2)| \leq C_{s,t} e^{\frac{(x-y)^2}{2(t-s)}}$ for all $x,y \in \mathbb{R}$.
\end{lem}

\begin{proof}
    This follows from Stirling's approximation since the entries of $Q^n(z,z')$ are given in terms of Binomial coefficients. We do not provide the details as they are standard (see Lemma 5.1 in \cite{Rah} for a similar detailed argument).
\end{proof}

\begin{lem} \label{lem:Smnlimit}
    Suppose $s >0$. Let $n_1 = Ns$, $z_1 = -2Ns - \sqrt{2N}x$ and $z_2 = -\sqrt{2N}y$. Then,
    $$ \sqrt{2N} S_{m,-n_1}(z_1,z_2) = (N/2)^{m/2} S_{m, -s}(x,y) \cdot (1 + O(N^{-1/2})).$$
  The constant in the error term $O(N^{-1/2})$ is uniformly bounded in $N \geq 1$ and $x,y \in \R$.
\end{lem}

\begin{proof}
We have that
    \begin{align*}
        S_{m,-n_1}(z_1,z_2) &= \frac{1}{2 \pi \mathbf{i}} \oint_{\gamma} \frac{dw}{w}\, (2w)^{z_2-z_1} (\frac{1-w}{w})^{n_1} \phi_m(w) \\
        &= \frac{1}{2 \pi \mathbf{i}} \oint_{\gamma} \frac{dw}{w}\, e^{Ns F(w) + \sqrt{2N} G(w)} \phi_m(w).
    \end{align*}
    Here,
    $$F(w) = \log(1-w) +\log w + 2\log(2); \quad G(w) = (x-y)\log(2w).$$
    Note that there is no pole at $w = 0$ when $N$ is sufficiently large.
    So we can deform the contour $\gamma$ to contain only the poles at $w = a_i$.
    
    Choose the contour $\gamma$ in the definition of $S_{m,-n_1}$ as follows. Let $d = 1 + \max_{i} |\mu_i|$.
    Let $$w = w(\theta) = \frac{1}{2} + \frac{de^{\mathbf{i}\theta}}{2 \sqrt{2N}}, \quad |\theta| \leq \pi.$$
    So $\gamma$ is a circular contour centred at $1/2$ and with radius $c_3d/\sqrt{N}$ (which therefore contains all the $a_i$).

    The critical point of $F$ is at $w = 1/2$:
    $$F(1/2) = 0; F'(1/2) = 0; F''(1/2) = -8.$$
    Consider the rescaling
    \begin{equation} \label{eqn:wz2}
    w = \frac{1}{2} + \frac{z}{2\sqrt{2N}}.
    \end{equation}
    We have that $|z| = d$ along $\gamma$ so that the $w$-contour becomes the circular contour $|z|=d$ in the variable $z$.
    Then,
    $$F(w) = -\frac{1}{2N}z^2 + O(d^3N^{-3/2}), \quad G(w) = (x-y) \frac{z}{\sqrt{2N}} + O(d^2 N^{-1}).$$
    We also have that $\sqrt{2N} dw/w = dz (1 + O(N^{-1/2})$. Finally, a computation shows that
    $$\phi_m(w) = (N/2)^{m/2} \times \prod_{i=1}^m (z-\mu_i)^{-1} \cdot (1 + O(N^{-1/2}).$$

    The lemma follows from these estimates.
\end{proof}

\begin{lem} \label{lem:Sbarlimit}
    Suppose $s < t$. Let $n = N(t-s)$, $z_1 = -2Ns - x\sqrt{2N}$ and $z_2 = -2Nt -y\sqrt{2N}$. Then,
    $$ \sqrt{2N} \bar{S}_{m,n}(z_1,z_2) = (N/2)^{-m/2} \bar{S}_{m,(t-s)}(x,y)\cdot (1 + O(N^{-1/2})).$$
  The constant in the error term $O(N^{-1/2})$ is uniformly bounded in $N \geq 1$ and $x,y \in \R$.
\end{lem}

\begin{proof}
We have that
    \begin{align*}
        \bar{S}_{m,n}(z_1,z_2) &= \frac{1}{2 \pi \mathbf{i}} \oint_{\gamma} \frac{dw}{1-w} (2(1-w))^{z_2-z_1} (\frac{1-w}{w})^n \phi_m(1-w)^{-1} \\
        &= \frac{1}{2 \pi \mathbf{i}} \oint_{\gamma} \frac{dw}{1-w} e^{N(t-s)F(w) + \sqrt{2N}G(w)} \phi_m(1-w)^{-1},
    \end{align*}
    where
    $$F(w) = -\log(1-w) -\log w -2\log(2); \quad G(w) = (x-y)\log (2(1-w)).$$
    The contour $\gamma$ is chosen to be the circle $\{|w|=1/2\}$; we parametrize it as
    $$ w = w(\theta) = \frac{1}{2} e^{\mathbf{i}\theta/\sqrt{2N}}, \quad |\theta| \leq \pi \sqrt{2N}.$$

    The critical point of $F$ is at $w = 1/2$:
    $$F(1/2) = 0; F'(1/2) = 0; F''(1/2) = 8.$$
    Locally, if $|\theta| \leq L$, then $w$ has the form
    \begin{equation} \label{eqn:wz3}
    w = \frac{1}{2}\left (1 + \frac{z}{\sqrt{2N}}\right ); \quad 1-w = \frac{1}{2} \left (1 - \frac{z}{\sqrt{2N}}\right) \end{equation}
    with $|z| \leq L$. Then, assuming that $|z| \leq L$,
    $$F(w) = \frac{1}{2N}z^2 + O_L(N^{-3/2}), \quad G(w) = -(x-y) \frac{z}{\sqrt{2N}} + O_L(N^{-1}).$$
    We also have that $\sqrt{2N} dw/(1-w) = dz (1 + O(N^{-1/2})$. Finally, a computation shows that
    $$\phi_m(1-w)^{-1} = (N/2)^{-m/2} \prod_{i=1}^m (-z-\mu_i) \cdot (1 + O(N^{-1/2})).$$
    Locally, the contour $\gamma$ becomes the vertical line $\Re(z) = 0$ oriented upwards (the tangent at $w = 1/2)$.

    Globally, we have $\Re(F(w)) = -\log(2|1-w|)$, which equals $- \log(5-4\cos(\theta/\sqrt{2N}))$ for $|\theta| \leq \pi \sqrt{2N}$.
    An exercise shows that $-\log(5-4\cos(x)) \leq - \frac{1}{100}x^2$ for $|x| \leq \pi$. As such, $\Re(F(w)) \leq - \frac{\theta^2}{200 N}$
    for all $\theta$. From this it is easy to verify that the integrand is bounded by
    $$ e^{- \frac{(t-s)}{200}\theta^2 + 10 |x-y| |\theta| + C_{\mu_1,\cdots, \mu_m} \log(1+|\theta|)}.$$

    By the dominated convergence theorem, and changing variables $z \mapsto -z$, we get to the assertion of the lemma.
    
\end{proof}

\begin{lem} \label{lem:Sepilimit}
    Let $t > 0$. Suppose $n_2 = Nt$, $z_1 = -\sqrt{2N}x$, $z_2 = -2Nt -\sqrt{2N}y$. Then,
    $$ S^{epi(\tilde{x})}_{m,n_2}(z_1,z_2) = (N/2)^{-m/2} S^{hypo(b)}_{m,t}(x,y) \cdot (1 + O(N^{-1/2}).$$
\end{lem}

\begin{proof}
    Define the process $B^N(s) = - \frac{B_{\lfloor Ns \rfloor} + 2Ns}{\sqrt{2N}}$ for $s \in [0,1]$. By Donsker's Theorem, there is a coupling of the processes $B^N$ for every $N$ together with a Brownian motion $B(s)$ such that $B^N(s) \to B(s)$ uniformly on compacts, almost surely.

    Recall the hitting time $\tau = \tau_N$, namely,
    $$\tau_N = \min \{n=0,1,\ldots: B_n > \tilde{x}_{n+1}\}.$$
    Let $\tau^N = \tau_N/N$, which converges almost surely to
    $$\tau = \inf \{s \in \geq 0: B(s) \leq b(s)\}$$
    under the aforementioned coupling.

    We find that
    $$S^{epi(y)}_{m,n_2}(z_1,z_2) = \mathbf{E}\left[ \bar{S}_{m,N(t-\tau^N)}(-2N \tau^N -\sqrt{2N} B^n(\tau^N), -2Nt - \sqrt{2N}y) \mathbf{1}_{\tau^N < t} \mid B^N(0)=x\right].$$
    The conclusion now follows from Lemma \ref{lem:Sbarlimit} above, the almost sure convergence of the aforementioned random walks to Brownian motion and continuity of $S_{t,m}(x,y)$ in the parameters $t \geq 0$ and $x,y \in \mathbb{R}$.
\end{proof}

The lemmas \ref{lem:Qlimit}, \ref{lem:Smnlimit}, \ref{lem:Sbarlimit} and \ref{lem:Sepilimit} provide the pointwise convergence of kernels comprising the limit $K$.
To complete the proof, we have to establish decay estimates so that the Fredholm series expansion converges absolutely.
It is enough to show, by way of Hadamard's inequality, that on the $(i,j)$-block of the kernel $K$ we have an estimate of the form $|K(i,x;j,y)| \leq f_i(x) g_j(y)$
where $f_i$ is bounded over $[a_i,\infty)$ and $g_j$ is integrable over $[a_j,\infty)$.
The errors provided by the aforementioned lemmas are uniformly bounded in the variables $x,y$ of the kernel.
So it is enough to establish decay estimates for the kernels $e^{(t-s)\partial^2/2}$ and $S_{B,n,-t} \cdot S^{hypo(b)}_{B,n,s}$.
This will also confirm that $K$ has an absolutely convergent Fredholm series.

In order to get decay estimates, we have to conjugate the kernel $K$. The conjugation factor we need for the $(i,j)$ block is
\begin{equation} \label{eqn:conjugation}
    e^{-\kappa_i |x| + \kappa_j |y|}
\end{equation}
where the constants $\kappa_i$ are sufficiently large and have to satisfy $\kappa_i > \kappa_j$ when $i < j$.
For example, we may choose $\kappa_i = C -i$ for a large constant $C$ that depends on $t_1, \ldots, t_k$ and $\mu_1, \ldots, \mu_m$.

With this conjugation factor, arguing as in the proof of \cite[Lemma 4.5]{Rah}, it follows that
$$e^{-\kappa_i |x| + \kappa_j |y|} \, e^{\frac{t_j-t_i}{2} \partial^2}(x,y) \ind{t_i < t_j} \leq f_i(x) g_j(y)$$
where $f_i$ is bounded and $g_j$ is integrable over $\R$.

\begin{lem} \label{lem:Hermitedecay}
    Suppose $m \geq 1$ and $t \in [0,T]$. There is a constant $C = C_{m,T, \mu_1,\ldots, \mu_m}$ such that
    $$ |S_{m,-t}(x,y)| \leq e^{C(|x-y|+1)} \quad \text{and} \quad |\bar{S}_{m,t}(u,v)| \leq C (|x-y|^n+1) e^{\frac{t}{2}\partial^2}(x,y).$$
\end{lem}

\begin{proof}
    Note that $\bar{S}_{m,t} = \prod_{i=1}^m (\partial - \mu_i) \cdot e^{t\partial^2/2}$. It follows from this that the kernel of $\bar{S}_{m,t}$
    is a linear combination of Hermite polynomials of degree at most $m$ multiplied by the heat kernel $e^{t\partial^2/2}$. The bound follows
    from this.

    Now consider $S_{m,-t}(x,y)$. Let $\mu_{min} = \min_i \mu_i$ and $\mu_{max} = \max_i \mu_i$.
    Choose the contour $\gamma$ in its definition to be a rectangle that intersects the real axis at the points $-1+\mu_{min}$ and $1 + \mu_{max}$
    and has imaginary part equal to $\pm 1$ along the horizontal sides. Then, $|\prod_{i=1}^m z-\mu_i| \geq 1$. The exponential factor in the integrand
    is easily seen to be bounded by $\exp \{\text{constant} + |x-y| (1+\max \{\mu_{max},-\mu_{min}\})\}$. As the contour of integration has length
    $6 + 2(\mu_{max}-\mu_{min})$, the bound follows.
\end{proof}

\begin{prop} \label{prop:SSdecay}
    Let $a \in \R$ and $0 < s,t \leq T$ and $m \geq 1$. Set $M_b = \max_{y \in [0,T]} |b(y)|$.
    There is a constant $C = C_{a, T, M_b, m, \mu_i}$ such that for all $x \in \R$ and $y \geq a$,
    $$ |S_{m,-s} \cdot S^{hypo(b)}_{m,t}(u,v) | \leq e^{C(|x|+1)} (|y|^m + 1)e^{-\frac{y^2}{4T}}.$$
\end{prop}

The proof of this Proposition is given in \cite[Proposition 5.1]{Rah} with minor changes to the argument.
One needs to use the bounds from Lemma \ref{lem:Hermitedecay} and follow the aforementioned proof, and so we omit the details.

Proposition \ref{prop:SSdecay} shows that if we conjugate $S_{m,-t_i} \cdot S^{hypo(b)}_{m,t_j}$ by the factor in \eqref{eqn:conjugation},
then it is bounded in absolute value by $f_i(x) g_j(y)$ where $f_i$ is bounded over $\R$
and $g_j$ is integrable over $[a_j,\infty)$. So the Fredholm series expansion of $\dt{I-\chi_aK\chi_a}$ with the conjugation is absolutely convergent.
This completes the proof of Theorem \ref{thm:BLPP}.

\subsection{Computation of some kernels: the narrow wedge and flat boundaries}

Firstly, consider the so called narrow wedge boundary, whereby $b_{nw}(0) = 0$ and $b_{nw}(t) = -\infty$ for $t > 0$.
We can approximate it with the continuos functions $b_L(t) = -L t$ in the limit $L \to \infty$.
Under this approximation, the hitting times $\tau_L = \inf \{ s \geq 0: W(s) \leq - Ls\}$ converge, monotonically and almost surely, to
the hitting time $\tau = \inf\{s \geq0: W(s) \leq b_{nw}(s)\}$. Note that $\tau$ is always $\infty$ unless $W(0) \leq b_{nw}(0) = 0$,
in which case $\tau = 0$. Therefore, $$S^{\mathrm{hypo}(b_{nw})}_{m,t}(x,y) = \ind{x \leq 0} \bar{S}_{m,t}(x,y).$$

Consequently,
\begin{align*}
    S_{m,-t_1} \cdot S^{\mathrm{hypo}(b_{nw})}_{m,t_2}(x,y) &= \frac{1}{(2 \pi \mathbf{i})^2} \int_{-\infty}^0 du \oint_{\gamma} dw \oint_{\Gamma} dz \\
    & e^{u(z-w)}\, e^{- \frac{t_1}{2}w^2 + xw} e^{\frac{t_2}{2}z^2-yz}\prod_{i=1}^m \frac{z-\mu_i}{w-\mu_i} \\
    &= \frac{1}{(2 \pi \mathbf{i})^2} \oint_{\gamma} dw \oint_{\Gamma} dz \\
    & \frac{e^{\frac{t_2}{2}z^2-yz}}{e^{\frac{t_1}{2}w^2 - xw}} \prod_{i=1}^m \frac{z-\mu_i}{w-\mu_i} \cdot \frac{1}{z-w}.\\
\end{align*}
Here we need to arrange the contours so that $\Re(z-w) > 0$ always, so $\Gamma$ lies to the right of $\gamma$ ($\Gamma > \gamma$).

Define the kernel
\begin{equation} \label{eqn:Knw}
    K_{nw}(t_1,x;t_2,y) = \frac{1}{(2 \pi \mathbf{i})^2} \oint_{\gamma} dw \oint_{\Gamma} dz
    \frac{e^{\frac{t_2}{2}z^2-yz}}{e^{\frac{t_1}{2}w^2 - xw}} \prod_{i=1}^m \frac{z-\mu_i}{w-\mu_i} \cdot \frac{1}{z-w}
\end{equation}
where $\gamma$ is a closed contour containing all the poles at $w=\mu_i$ and $\Gamma$ is a vertical contour that lies to the right of $\gamma$.

We have established
\begin{prop} \label{prop:nw}
    Let $BLPP((0,0); (m,t)) = \max_{0=t_0 \leq t_1 \leq \cdots \leq t_m = t} \sum_{k=1}^m B^{\mu}_k(t_k) - B^{\mu}_k(t_{k-1})$
    where $B^{\mu}_k$ are independent Brownian motions with respective drift $\mu_k$. Then, for $m \geq 1$ and $0 < t_1 < \cdots < t_k$,
    $$ \mathbf{Pr}(BLPP((0,0); (t_i,m)) \leq a_i, 1 \leq i \leq k) = \det(I -\chi_a K \chi_a)_{L^2(\{t_1,\ldots, t_k\}\times \R)}$$
    where
    $$ K(t_i,x;t_j,y) = - e^{(t_j-t_i)\partial^2/2}\ind{t_i < t_j} + K_{nw}(t_i,x;t_j,y)$$
    and $\chi_a(t_i,x) = \ind{x \geq a_i}$.
\end{prop}

Now we turn to the computation of the kernel for the flat boundary $b \equiv 0$. We have
$S_{m,-t_1} S^{\mathrm{hypo}(0)}_{m,t_2} = S_{m,-t_1} \chi_{(-\infty,0]}S^{\mathrm{hypo}(0)}_{m,t_2} + S_{m,-t_1} \chi_{(0,\infty)} S^{\mathrm{hypo}(0)}_{m,t_2}$.
But when $x \leq 0$, $S^{\mathrm{hypo}(b)}_{m,t}(x,y) = \bar{S}_{m,t}(x,y)$ because then the hitting time $\tau = 0$. Therefore,
\begin{equation} \label{eqn:SflatA} S_{m,-t_1} \chi_{(-\infty,0]}S^{\mathrm{hypo}(0)}_{m,t_2}(x,y) = K_{nw}(t_1,x;t_2,y).\end{equation}
Next, consider $S_{m,-t_1} \chi_{(0,\infty)} S^{\mathrm{hypo}(0)}_{m,t_2}$. Recall that
$$ S^{\mathrm{hypo}(b)}_{m,t}(x,y) = \prod_{i=1}^m (-\partial_y-\mu_i) \cdot S_t^{\mathrm{hit}(b)}(x,y)$$
where $S_t^{\mathrm{hit}(b)}(x,y) = \mathbf{Pr}(\tau \le t, W(t) \in dy\mid W(0)=x)$.
For the flat boundary, by the reflection principle, we find that for $x > 0$,
$$S_t^{\mathrm{hit}(0)}(x,y) = e^{t\partial^2/2}(x,y) \ind{y \leq 0} + e^{t \partial^2/2}(-x,y) \ind{y >0}.$$
Upon writing the heat kernel as a contour integral, we thus find that
\begin{align*}
    S^{\mathrm{hypo}(0)}_{m,t}(x,y) &= \frac{\ind{y \leq 0}}{2 \pi \mathbf{i}} \oint_{\Gamma} dz\, e^{\frac{t}{2}z^2 + (x-y)z} \prod_{i=1}^m (z-\mu_i)
    + \frac{\ind{y > 0}}{2 \pi \mathbf{i}} \oint_{\Gamma} dz\, e^{\frac{t}{2}z^2 + (-x-y)z} \prod_{i=1}^m (z-\mu_i).
\end{align*}
A computation now gives
\begin{equation} \label{eqn:SflatB}
    S_{m,-t_1} \chi_{(0,\infty)} S^{\mathrm{hypo}(0)}_{m,t_2}(x,y) = A_1(x,y) + A_2(x,y)
\end{equation}
where
\begin{align*}
    A_1(x,y) &= - \frac{\ind{y \leq 0}}{(2 \pi \mathbf{i})^2} \oint_{\gamma} dw \oint_{\Gamma} dz
    \frac{e^{\frac{t_2}{2}z^2-yz}}{e^{\frac{t_1}{2}w^2 - xw}} \prod_{i=1}^m \frac{z-\mu_i}{w-\mu_i} \cdot \frac{1}{z-w}; \\
    A_2(x,y) &= \frac{\ind{y > 0}}{(2\pi \mathbf{i})^2} \oint_{\gamma} dw \oint_{\Gamma} dz
    \frac{e^{\frac{t_2}{2}z^2-yz}}{e^{\frac{t_1}{2}w^2 - xw}} \prod_{i=1}^m \frac{z-\mu_i}{w-\mu_i} \cdot \frac{1}{z+w}.
\end{align*}
In the integral for $A_1$, the vertical contour $\Gamma$ is to the right of the closed contour $\gamma$ enclosing all the poles at $w=\mu_i$ ($\Gamma > \gamma$ so that $\Re(z-w)>0$ always). Similarly, in the integral for $A_2$, the vertical contour $\Gamma$ is to the right of the closed contour $-\gamma$ ($\Gamma > -\gamma$ so that $\Re(z+w)>0$ always).

Combining \eqref{eqn:SflatB} with \eqref{eqn:SflatA} we deduce that
\begin{align} \label{eqn:Kflat}
    S_{m,-t_1} \cdot S^{\mathrm{hypo}(0)}_{m,t_2}(x,y) &= \frac{\ind{y > 0}}{(2 \pi \mathbf{i})^2} \oint_{\gamma} dw \oint_{\Gamma} dz
    \frac{e^{\frac{t_2}{2}z^2-yz}}{e^{\frac{t_1}{2}w^2 - xw}} \prod_{i=1}^m \frac{z-\mu_i}{w-\mu_i} \cdot \frac{1}{z-w} \\ \nonumber
    & +
    \frac{\ind{y > 0}}{(2\pi \mathbf{i})^2} \oint_{\gamma} dw \oint_{\Gamma} dz
    \frac{e^{\frac{t_2}{2}z^2-yz}}{e^{\frac{t_1}{2}w^2 - xw}} \prod_{i=1}^m \frac{z-\mu_i}{w-\mu_i} \cdot \frac{1}{z+w}
\end{align}
with the contours arranged as indicated above.

We have thus established
\begin{prop} \label{prop:flat}
    Let $BLPP(0; (t,m)) = \max_{0\leq t_0 \leq t_1 \leq \cdots \leq t_m = t} \sum_{k=1}^m B^{\mu}_k(t_k) - B^{\mu}_k(t_{k-1})$
    where $B^{\mu}_k$ are independent Brownian motions with drift $\mu_k$. Then, for $m \geq 1$ and $0 < t_1 < \cdots < t_k$,
    $$ \mathbf{Pr}(BLPP(0; (t_i,m)) \leq a_i, 1 \leq i \leq k) = \det(I -\chi_a K \chi_a)_{L^2(\{t_1,\ldots, t_k\}\times \R)}$$
    where
    $$ K(t_i,x;t_j,y) = - e^{(t_j-t_i)\partial^2/2}\ind{t_i < t_j} + K_{flat}(t_i,x;t_j,y)$$
    with $K_{flat}$ given by \eqref{eqn:Kflat} and $\chi_a(t_i,x) = \ind{x \geq a_i}$.
\end{prop}

\section{Proof of Theorem \ref{thm:arithmetic}} \label{sec:arithmetic}
Let $\Delta > 0$ be fixed. For $n \geq 1$, consider the arithmetic progressions $$\mu^n_i = \mu^n_1 -\Delta \cdot (i-1) \;\;\text{for}\;\; 1 \leq i \leq n.$$

Consider, for each $n$, the Brownian last passage model $BLPP((0,0);(t,n))$ with drifts given by $\mu^n_i$.
According to \cite[Proposition 4.2]{OCJ}, for the $n \times n$ random Hermitian matrix
$$ H = \mathrm{diag}(\mu^n_i) + H^{\mathrm{GUE}},$$
the law of its largest eigenvalue $\lambda_{max}(H)$ satisfies
$$ \lambda_{max}(H) \stackrel{law}{=} BLPP((0,0);(1,n)).$$
Define the random variable $\chi_n$ according to
$$ BLPP((0,0);(1,n)) = \mu^n_1 + \frac{\log(n-1) + \chi_n}{\Delta}.$$
Theorem \ref{thm:arithmetic} thus follows from

\begin{thm} \label{thm:blpparithmetic}
    For $a \in \R$, as $n \to \infty$,
    $$\mathbf{Pr}(\chi_n \leq a) \to \det(I-K_{\Delta})_{L^2[a,\infty)}$$
    where
    $$K_{\Delta}(x,y) = \frac{1}{(2 \pi \mathbf{i})^2} \oint_{\gamma_{rec}}d\zeta \oint_{\gamma_{ver}} dz\,
    \frac{e^{\frac{\Delta^2}{2} z^2 -yz}}{e^{\frac{\Delta^2}{2}\zeta^2-x\zeta}} \cdot \frac{\Gamma(\zeta)}{\Gamma(z)}\cdot \frac{1}{z-\zeta}.$$
    Here $\gamma_{ver}$ is the vertical contour $\{\Re(z) = 1\}$ oriented upwards, and $\gamma_{rec}$ is the counter clockwise oriented contour
    $\{t \pm \mathbf{i}/2; t \leq 1/2\} \cup \{1/2 + \mathbf{i}t; |t| \leq 1/2\}$.
\end{thm}

\begin{proof}
Recall Proposition \ref{prop:nw} and the kernel $K_{nw}$ which governs the law of $BLPP((0,0);(t,n))$ with drifts $\mu^n_i$.
We find after rescaling the kernel that
$$ \mathbf{Pr}(\chi_n \leq a)= \det(I - K_n)_{L^2[a,\infty)}$$
where $K_n(x,y) = \Delta^{-1} K_{nw}(1,\mu^n_i + (\log (n-1) + x)/\Delta;1;\mu^n_1+(\log (n-1) + y)/\Delta).$
Changing variables $z \mapsto \Delta z + \mu^n_1$ and $w \mapsto \Delta \zeta + \mu^n_1$ in the kernel $K_{nw}$
then shows that
$$K_n(x,y) = \frac{e^{\frac{\mu^n_1}{\Delta}(x-y)}}{(2 \pi \mathbf{i})^2} \oint_{\gamma_a}d\zeta \oint_{\gamma_b} dz\,
\frac{e^{\frac{\Delta^2}{2}z^2-yz}}{e^{\frac{\Delta^2}{2} \zeta^2-x\zeta}} \frac{e^{-z\log(n-1)}}{e^{-\zeta \log(n-1)}} \prod_{i=1}^n \frac{z+i-1}{\zeta+i-1} \frac{1}{z-\zeta}.$$
Here $\gamma_a$ is a contour containing all the $\zeta$-poles and $\gamma_b$ is a vertical contour to the right of $\gamma_a$.

The term $e^{\frac{\mu^n_1}{\Delta}(x-y)}$ is a conjugation factor and we remove it from the kernel without affecting the Fredholm determinant of $K_n$.

The product can be written as $$\prod_{i=1}^n \frac{z+i-1}{\zeta+i-1} = \frac{z}{\zeta} \prod_{i=1}^{n-1} \frac{1 + \frac{z}{i}}{1 + \frac{\zeta}{i}}$$
Recall that $\log n = H_n -\gamma_{EM} + O(1/n)$ where $H_n = 1 + \frac{1}{2} + \cdots + \frac{1}{n}$ is the harmonic sequence and $\gamma_{EM}$ is the Euler-Mascheroni constant.
Consequently, up to a multiplicative error of order $1 + O(1/n)$,
$$\frac{e^{-z\log(n-1)}}{e^{-\zeta \log(n-1)}} \prod_{i=1}^n \frac{z+i-1}{\zeta+i-1} = \frac{z e^{z \gamma_{EM}} \prod_{i=1}^{n-1} (1 + \frac{z}{i})e^{-z/i}}{\zeta e^{\zeta \gamma_{EM}} \prod_{i=1}^{n-1} (1 + \frac{\zeta}{i})e^{-\zeta/i}}$$
Recall the Weierstrass factorization theorem for the Gamma function:
\begin{equation} \label{eqn:Gammafactor} \frac{1}{\Gamma(z)} = z e^{\gamma_{EM} z} \prod_{i=1}^{\infty} (1 + \frac{z}{i}) e^{-z/i}.\end{equation}
The error rate of the infinite product is bounded by
\begin{equation} \label{eqn:Wfactor}
\left |\frac{1}{z e^{\gamma_{EM} z} \Gamma(z)} - \prod_{i=1}^{n-1} (1 + \frac{z}{i}) e^{-z/i} \right| \leq C|z|^2 \times \sum_{i=n}^{\infty} i^{-2}.
\end{equation}
We also have the identities
$$ | \Gamma(1+\mathbf{i}v)|^2 = \frac{\pi v}{\sinh{\pi v}}, \quad |\Gamma(-m+\mathbf{i}v)|^2 = \frac{\pi}{v \sinh{\pi v}} \prod_{k=1}^m (k^2+v^2)^{-1}.$$
From these identities we can deduce that
\begin{align}
  \label{eqn:verdecay}  |\Gamma(z)| &\geq 10 (|\Im(z)|+1), & z \in \gamma_{ver};\\
  \label{eqn:hordecay} |\Gamma(\zeta)| & \leq 10, & \zeta \in \gamma_{rec}.
\end{align}

We can choose the contour $\gamma_b$ to be the contour $\gamma_{ver} = \{\Re(z) = 1\}$. We choose the contour $\gamma_a$ to be a rectangle which intersects the real line at $1/2$ and $-L$ for any $L > n-1$, and its imaginary parts equal $\pm 1/2$ along the horizontal sides. By letting $L \to \infty$, we can turn $\gamma_a$ into $\gamma_{rec}$. Indeed, suppose
$\zeta = u+\mathbf{i}v$ where $u \leq -L$ and $|v| \leq 1/2$. Then the real part of $(\Delta^2/2) \zeta^2 - x\zeta$ is bounded below by $(\Delta^2/2) u^2 -xu -\Delta^2/8$.
From the estimates \eqref{eqn:Wfactor} and \eqref{eqn:hordecay}, we may deduce that
$$\left | \zeta e^{\zeta \gamma_{EM}} \prod_{i=1}^{n-1} (1 + \frac{\zeta}{i})e^{-\zeta/i} \right |^{-1} \leq e^{C (|u|+1)}.$$
Thus the $\zeta$-integral over the region $\{t \pm \mathbf{i}/2; t \leq -L\} \cup \{-L + \mathbf{i}t; |t| \leq 1/2\}\}$ is bounded in modulus by
$$ 2 \int_{-\infty}^{-L} du\, e^{-(\Delta^2/2)u^2 + C|u| + C} + e^{-(\Delta^2/2)L^2 + CL + C}$$
which tends to zero as $L \to \infty$. As such
$$ K_n(x,y) = \frac{1}{(2\pi \mathbf{i})^2} \oint_{\gamma_{rec}} d\zeta \oint_{\gamma_{ver}} dz\, \frac{e^{\frac{\Delta^2}{2}z^2-yz}}{e^{\frac{\Delta^2}{2} \zeta^2-x\zeta}}
\frac{z e^{z \gamma_{EM}} \prod_{i=1}^{n-1} (1 + \frac{z}{i})e^{-z/i}}{\zeta e^{\zeta \gamma_{EM}} \prod_{i=1}^{n-1} (1 + \frac{\zeta}{i})e^{-\zeta/i}} \frac{1}{z-\zeta}.$$
So as $n \to \infty$, $K_n(x,y)$ converges pointwise to $K_{\Delta}(x,y)$ due to \eqref{eqn:Gammafactor}, \eqref{eqn:Wfactor} and the dominated convergence theorem.

In order to derive the convergence of Fredholm determinants, we need to decay estimates on $K_n(x,y)$ in terms of the parameters $x,y$.
By parametrising the contours $\gamma_{rec}$ and $\gamma_{ver}$, using \eqref{eqn:Wfactor}, \eqref{eqn:verdecay} and \eqref{eqn:hordecay}, we will find that
$$ |K_n(x,y)| \leq C e^{\frac{1}{2}x - y}$$
for some constant $C$. If we conjugate the kernel by the factor $e^{\frac{2}{3}(y-x)}$, the conjugated kernel obeys
$$ e^{\frac{2}{3}(y-x)} |K_n(x,y)| \leq C e^{-\frac{1}{6}x - \frac{1}{3}y},$$
which is bounded and integrable over $x,y \in [a, \infty)$. Thus, with this conjugation, we get convergence of the Fredholm determinants as required.
\end{proof}

\section{Proof of Theorem \ref{thm:airyconvergence}} \label{sec:airy}
Let $H(t)$ be Brownian motion in the space of $n \times n$ Hermitian matrices (started from zero, see \eqref{eqn:hbm}).
Let $H_0$ be a fixed $n \times n$ Hermitian matrix. Consider the process
$$ M(t) = H(t) + tH_0$$
and its largest eigenvalue $\lambda_{max}(M(t))$. If $H_0$ has spectral decomposition $H_0 = U^* \Lambda U$ where $\Lambda = \mathrm{diag}(\lambda)$ is the diagonal matrix of eigenvalues and $U$ is unitary, then $U M(t) U^* = UH(t)U^* + t \Lambda$ has the same eigenvalues as $M(t)$. Since $U H(t) U^*$ has the same law of $H(t)$, it follows that
$$ \lambda_{max}(M(t)) \stackrel{law}{=} \lambda_{max}(H(t) + t \Lambda).$$
The eigenvalues of $H(t) + t\Lambda$ have the same law as $n$ independent Brownian motions with drifts given by the eigenvalues
of $\Lambda$, conditioned not to collide on $(0,\infty)$ \cite[Section 3.5.1]{AOCW}.
In particular $\lambda_{max}(H(t) + t \Lambda)$ is equal in law to the top particle among these $n$ Brownian motions conditioned not to collide.
It is shown in \cite[Theorem 8.3]{OC} (see also \cite{BBOC, BJ, OY}) that said top particle has the same law as the Brownian last passage process $t \mapsto BLPP((0,0); (t,n))$ where the Brownian motions
in the last passage problem have drifts given by the eigenvalues of $\Lambda$. In summary, if $H_0$ has eigenvalues $\lambda_1, \ldots, \lambda_n$ then,
as processes in $t$,
\begin{equation} \label{eqn:BBOC}
    \lambda_{max}(H(t) + t H_0) \stackrel{law}{=} \max_{0=t_0 \leq t_1 \leq \cdots \leq t_n = t} \sum_{k=1}^n B^{\lambda}_k(t_k) - B^{\lambda}_k(t_{k-1})
\end{equation}
where $B^{\lambda}_k$ are independent Brownian motions with corresponding drifts $\lambda_k$.

Now consider the process
$$X(t) = H(t) + H_0$$
and its largest eigenvalue $\lambda_{max}(X(t))$. Time inversion $$ B(t) \mapsto tB(1/t)$$ leaves the law of a standard Brownian motion invariant. Consequently, under time inversion, $H(t)$ does not change in law. Under time inversion, $M(t)$ is mapped to
$$ M(t) \mapsto tM(1/t) = t H(1/t) + H_0 \stackrel{law}{=} X(t).$$
Consequently, by \eqref{eqn:BBOC}, the finite dimensional distributions of $\lambda_{max}(X(t))$ satisfy
\begin{equation}
    \mathbf{Pr}(\lambda_{max}(X(t_i) \leq a_i, 1 \leq i \leq k) = \mathbf{Pr}(BLPP((0,0);(1/t_i,n)) \leq a_i/t_i, 1 \leq i \leq k).
\end{equation}
Proposition \eqref{prop:nw} now provides the following formula.
\begin{prop} \label{prop:Hermitiankernel}
    Let $H_0$ be a $n \times n$ Hermitian matrix with eigenvalues $\lambda_1, \ldots, \lambda_n$, and
    $H(t)$ be a Brownian motion in the space of $n \times n$ Hermitian matrices started from zero as in \eqref{eqn:hbm}.
    Suppose $0 < t_1 < t_2 < \cdots < t_k$. Then
    $$ \mathbf{Pr}(\lambda_{max}(H(t_i)+H_0) \leq a_i, 1 \leq i \leq k) = \det(I - K)_{L^2(\{1, \cdots, k\} \times [0,\infty))}.$$
    The kernel $K$ takes the form
    $$K(i,x;j,y) = - e^{(t_j^{-1} - t_i^{-1})\partial^2/2}\left(x + \frac{a_i}{t_i}, y + \frac{a_j}{t_j}\right) \ind{t_j < t_i} 
    + K_{nw}\left (t_i^{-1}, x + \frac{a_i}{t_i}; t_j^{-1}, y + \frac{a_j}{t_j}\right)$$
    where $K_{nw}$ is the kernel from \eqref{eqn:Knw} with $\mu_i = \lambda_i$.
\end{prop}

Recall from \eqref{eqn:lambdant} that
$$ \bar{\lambda}_n(t) = \frac{\lambda_n(t) - a_n - 2t d_n^2(b_n-a_n)n^{-1/3}}{d_n n^{-2/3}}.$$
It is enough to show that the finite dimensional laws of $\bar{\lambda}_n$ converges to those of the Airy process $\mathcal{A}$.
Recall the extended Airy kernel $K_{Airy}$  from \eqref{eqn:airykernel}.
Given $\tau_1 < \cdots < \tau_k$, we must prove that
$$ \mathbf{Pr}(\bar{\lambda}_n(\tau_i) \leq \xi_i, 1 \leq i \leq k) \to \det(I - K)_{L^2(\{1,\ldots,k\} \times [0,\infty)])}$$
where $K(i,x;j,y) = K_{Airy}(\tau_i, x+\xi_i + \tau_i^2; \tau_j, y+\xi_j + \tau_j^2)$.

By Proposition \ref{prop:Hermitiankernel}, we find that
$$\mathbf{Pr}(\bar{\lambda}_n(\tau_i) \leq \xi_i; 1 \leq i \leq k) = \det(I- K_n)_{L^2(\{1,\ldots,k\}\times [0,\infty))}$$
where $K_n(i,x;j,y) = d_n n^{1/3} K(i, d_n n^{1/3}x; j, d_n n^{1/3}y)$ and $K$ is the kernel presented there
with choice of parameters $t_i = \frac{1-2d_n^2\tau_i n^{-1/3}}{n}$ and $a_i = a_n + 2\tau_i d_n^2 (b_n-a_n) n^{-1/3} + d_n \xi_i n^{-2/3}$,
and $\lambda_j = \nu^n_j$. Assume $n$ is such that $\nu^n \in F(\alpha,\beta)$ where $0 < \alpha, \beta < \infty$.

A calculation shows that for an explicit, positive constant $c_n(i,x)$ (see below for its definition),
$$ K_n(i,x;j,y) = - \frac{c_n(j,y)}{c_n(i,x)}\, e^{(\tau_j-\tau_i) \partial^2}(x+ \xi_i, y+ \xi_j) \ind{\tau_j > \tau_i} (1 + O(n^{-1/3})) + J'_n(i,x;j,y).$$
The kernel $J'_n$ is expressed as a double contour integral as follows.
\begin{align*}
    J'_n(i,x;j,y) &= d_n n^{1/3} \oint_{\gamma_{ver}} dw \oint_{\gamma_{rec}}dz\,
    \frac{e^{n f_n(w) + n^{2/3}d_n^2 \tau_j g_n(w) - n^{1/3}d_nh_n(w,j,y)+\varepsilon(w)}}{e^{n f_n(z) + n^{2/3}d_n^2 \tau_i g_n(z) - n^{1/3}d_nh_n(z,i,x)+\varepsilon(z)}} \frac{1}{w-z}\, ,\\
    f_n(w) &= \frac{w^2}{2} - a_n w + \frac{1}{n} \sum_{i=1}^n \log(w-\nu^n_i), \\
    g_n(w) &= w^2 - 2b_n w,\\
    h_n(w,j,y) &= w( y+a_j + 4\tau_j^2d_n^3(b_n - \frac{w}{2})),\\
    |\varepsilon(w)| &\leq O_{\alpha,\beta,\tau_j}(|w|^2).
\end{align*}
The contour $\gamma_{rec}$ encloses all the poles at $z = \nu^n_i$ and $\gamma_{ver}$ is a vertical line lying to the right of $\gamma_{rec}$.

The numbers $a_n, b_n$ and $d_n$ are chosen so that
$$ f_n'(b_n) = f_n''(b_n) = 0,\; \frac{1}{3!} f'''_n(b_n) = \frac{d_n^3}{3}.$$
Note also that $g_n'(b_n) = 0$. Let us change variables $w \mapsto \frac{w}{d_n} + b_n$ and $z \mapsto \frac{z}{d_n}+b_n$. Then,
$$J'_n(i,x;j,y) = \frac{c_n(j,y)}{c_n(i,x)} J_n(i,x;j,y) $$
where
$$c_n(i,x) = e^{n^{2/3}d_n^2\tau_i g_n(b_n) - n^{1/3}d_n h_n(b_n,i,x)}$$
and
$$J_n(i,x;j,y) = n^{1/3} \oint_{\gamma_{ver}} dw \oint_{\gamma_{rec}}dz\,
    \frac{e^{n  A(w) + n^{2/3}d_n^2 B(w) - n^{1/3}d_nC(w,j,y) + \varepsilon(w)}}
    {e^{n A(z) + n^{2/3}d_n^2 B(z) - n^{1/3}d_n C(z,i,x) + \varepsilon(z)}} \frac{1}{w-z}$$
    with
\begin{align}
    A(w) &= f_n(\frac{w}{d_n}+b_n)-f_n(b_n)\\
    B(w) &= \tau_j (g_n(\frac{w}{d_n}+b_n)-g_n(b_n)) \\
    C(w,j,y) &= h_n(\frac{w}{d_n}+b_n,j,y) - h_n(b_n,j,y)\\
    |\varepsilon(w)| &\leq O_{\alpha,\beta,\tau_j}(|w|^2).
\end{align}
The contour $\gamma_{ver}$ is a vertical line lying to the right of zero and $\gamma_{rec}$ lies to the left of zero and encloses the poles at $z = d_n(v^n_i-b_n) < 0$.
We can remove the conjugation factor $c_n(j,y)/c_n(i,x)$ without changing the Fredholm determinant.
Then, in order to prove the theorem, we need to analyse the asymptotic behaviour of the kernel $J_n$.

In order to find the asymptotics of $J_n$, we need to choose good contours of integration. Contours need to be chosen such that $f_n(w)$ and $f_n(z)$ have strong
decay along them. The good choice of contours is found in the proof of Theorem 3.1 in \cite{Jo3}. We should choose $\gamma_{ver}$ to be the vertical contour
$$w(t) = \delta_1 n^{-1/3} + \mathbf{i}t$$ for $t \in \R$ and $\delta_1 > \max \{|\tau_1|, \ldots, |\tau_k|\}$. The contour $\gamma_{rec}$ should be the wedge-shaped contour
\begin{equation} \label{eqn:gammawedge} z(t) = \delta_2 n^{-1/3}+ \begin{cases}
    e^{\mathbf{i} \pi/6}t & \text{for}\; t \leq 0\\
    e^{\mathbf{i}5\pi/6} t & \text{for}\; t \geq 0
\end{cases}\end{equation}
with $\max \{|\tau_1|, \ldots, |\tau_k|\} < \delta_2 < \delta_1$.

To see why this is so, let $r(t) = \Re(f_n(x(t) + \mathbf{i}y(t))$ where $x(0) = b = b(\nu)$ and $y(0)=0$. A computation then gives
\begin{align*} r'(t) &= \frac{1}{n} \sum_{j=1}^n \frac{(-x^2x' + 2yy'x+y^2x'+2bxx'-2byy'-b^2x')\nu_j}{(b - \nu_j)^2((x-\nu_j)^2 + y^2)} \\
& + \frac{1}{n} \sum_{j=1}^n \frac{((xx'-yy'-2bx')(x^2+y^2) + b^2(xx'+yy'))}{(b - \nu_j)^2((x-\nu_j)^2 + y^2)}.
\end{align*}
We want to choose a contour such that the numerator above does not depend on $\nu$. So we need
$$  -x^2x' + 2yy'x+y^2x'+2bxx'-2byy'-b^2x'= 0,$$
which implies
$$ -\frac{1}{3} x^3 +y^2x + b x^2 - by^2 - b^2 x = \text{constant}$$
and the constant will be $- b^2/3$ since $x(0) = b$ and $y(0)=0$. The above then factorises to
$$ y^2(x-b) = \frac{1}{3}(x-b)^3.$$
The choice $x \equiv b$ and $y(t) = t$ leads to the vertical contour $w(t)$ up to the translation by $\delta_1 n^{-1/3}$.
The choice $y = \pm \frac{1}{\sqrt{3}}(x-b)$ and $x(t) = t+b$ leads to the contour $z(t)$ up to the translation as well.

Along the vertical contour $w(t)$ we have the estimate, assuming $\nu \in F(\alpha,\beta)$,
\begin{equation} \label{eqn:wdecay}
    \Re(f_n(w(t)+b)-f_n(b) \leq \begin{cases}
        - \frac{t^4}{8\beta^2} & 0 \leq |t| \leq \beta, \\
        - \frac{\beta^2-2t^2}{8} & |t| \geq \beta.
    \end{cases}
\end{equation}
Along the wedge-shaped contour $z(t)$ we have the estimate, again assuming $\nu \in F(\alpha, \beta)$,
\begin{equation} \label{eqn:zdecay}
    f_n(b)-\Re(f_n(z(t)+b) \leq \begin{cases}
        - \frac{t^4}{24\beta^2} & 0 \leq |t| \leq \beta, \\
        - \frac{\beta^2-2t^2}{24} & |t| \geq \beta.
    \end{cases}
\end{equation}

Having chosen these contours, we rescale $z \mapsto n^{-1/3}z$ and $w \mapsto n^{-1/3}w$ to get
$$J_n(i,x;j,y) = \oint_{\gamma_{\delta}} dw \oint_{\gamma_{\theta}}dz\,
    \frac{e^{n (f_n(\frac{w}{d_n}n^{-1/3}+b_n)-f_n(b_n)) + \tau_j w^2 - w(y+\xi_j)-2\tau_j^2d_n^3w^2n^{-1/3} + \varepsilon(w^2n^{-2/3})}}
    {e^{n (f_n(\frac{z}{d_n}n^{-1/3}+b_n)-f_n(b_n)) + \tau_i z^2 - z(y+\xi_i)-2\tau_i^2d_n^3z^2n^{-1/3} + \varepsilon(z^2n^{-2/3})}} \frac{1}{w-z}.$$
Here $\gamma_{\delta}$ is the vertical contour $w(t) = \delta_1 + \mathbf{i}t$ oriented upwards and $\gamma_{\theta}$ is the wedge-shaped contour $\{ \delta_2+ e^{\pm \mathbf{i}\frac{\pi}{6}}t, t\leq 0\}$ oriented counter clockwise. We also have $\max_j |\tau_j| < \delta_2 < \delta_1$.

Let $$ L = L_n = n^{1/24}.$$
Suppose $|w| \leq L$ and $\nu^n \in F(\alpha,\beta)$. By Taylor's theorem (with remainder) we have that
$$ n (f_n\left(\frac{w}{d_n}n^{-1/3} + b_n\right) - f_n(b_n)) = \frac{1}{3}w^3 + O_{\alpha,\beta}(L^4 n^{-1/3}) = \frac{1}{3}w^3 + O_{\alpha,\beta}(n^{-1/6}).$$
Furthermore, $|\varepsilon(w^2 n^{-2/3})| = O_{\alpha,\beta,\tau_j}(n^{-7/12})$ and likewise for $|\varepsilon(z^2n^{-2/3})|$.
Similarly, $|2\tau_j^2d_n^3w^2n^{-1/3}| = O_{\alpha,\beta,\tau_j}(n^{-1/4})$ and likewise for $|2\tau_i^2d_n^3z^2n^{-1/3}|$.
Therefore, for $|w|,|z| \leq L$ and $\nu \in F(\alpha,\beta)$,
\begin{align} \label{eqn:localairy}
&\frac{e^{n (f_n(\frac{w}{d_n}n^{-1/3}+b_n)-f_n(b_n)) + \tau_j w^2 - w(y+\xi_j)-2\tau_j^2d_n^3w^2n^{-1/3} + \varepsilon(w^2n^{-2/3})}}{e^{n (f_n(\frac{z}{d_n}n^{-1/3}+b_n)-f_n(b_n)) + \tau_i z^2 - z(y+\xi_i)-2\tau_i^2d_n^3z^2n^{-1/3} + \varepsilon(z^2n^{-2/3})}} =\\ \nonumber
& \frac{e^{\frac{1}{3}w^3 + \tau_j w^2 - (y+\xi_j)w}}{e^{\frac{1}{3}z^3 + \tau_i z^2 - (x+\xi_i)z}} \times (1 + O_{\alpha,\beta}(n^{-1/6})).
\end{align}

Consider the intervals $I_1 = (-\infty, -L]$, $I_2 = [-L,L]$ and $I_3  = [L,\infty)$. Let $\gamma_{k,\delta}$ denote the contour $\gamma_{\delta}$ restricted to $t \in I_k$
and likewise for $\gamma_{k,\theta}$. Define
$$ I_{j,k} = \oint_{\gamma_{j,\delta}} dw \oint_{\gamma_{k,\theta}}dz\,
\frac{e^{n (f_n(\frac{w}{d_n}n^{-1/3}+b_n)-f_n(b_n)) + \tau_j w^2 - w(y+\xi_j)-2\tau_j^2d_n^3w^2n^{-1/3} + \varepsilon(w^2n^{-2/3})}}
{e^{n (f_n(\frac{z}{d_n}n^{-1/3}+b_n)-f_n(b_n)) + \tau_i z^2 - z(y+\xi_i)-2\tau_j^2d_n^3z^2n^{-1/3} + \varepsilon(z^2n^{-2/3})}} \, \frac{1}{w-z}$$
We have that
$$ J_n(i,x;j,y) = \sum_{j,k=1}^3 I_{j,k}.$$
Assume $n_0$ is such that $\nu^n \in F(\alpha,\beta)$ for every $n \geq n_0$ and consider $n \geq n_0$.

We can now argue exactly as in the proof of Theorem 3.1 in \cite{Jo3} by using the estimates \eqref{eqn:wdecay} and \eqref{eqn:zdecay}.
Firstly, it follows from \eqref{eqn:localairy} that
$$ I_{2,2} \to J_{Airy}(i,x+\xi_i;j,y+\xi_j) \quad \text{as}\; n \to \infty.$$
The aforementioned proof shows that there are constants $C = C_{\alpha,\beta} < \infty$ and $c = c_{\alpha,\beta} > 0$ such that
for all $x,y \geq 0$,
\begin{align*}
    |I_{2,2}| &\leq C e^{-c(x+y)}, \\
    |I_{2,k}| &\leq C e^{-cn^{1/6}-x}, \quad k=1,3,\\
    |I_{j,k}| &\leq C e^{-cn^{1/6}-cn^{1/8}x}, \quad j=1,3 \;\text{and}\; \;k=1,2,3.
\end{align*}
These estimates imply $J_n(i,x;j,y) \to J_{Airy}(i,x+\xi_i;j,y+\xi_j)$ as $n \to \infty$. Furthermore, they imply that
$J_n$ satisfies the bound $|J_n(i,x;j,y)| \leq f_i(x)g_j(y)$ for all $i,j$ where $f_i$ is integrable and $g_j$ is bounded over $[0,\infty)$.
As a result, by the dominated convergence theorem and Hadamard's inequality, it follows that
$$ \det(I-K_n)_{L^2(\{1,\ldots,k\} \times [0,\infty))} \to \det(I-K)_{L^2(\{1,\ldots,k\} \times [0,\infty))}$$
with \begin{align*}
K(i,x;j,y) &= - e^{(\tau_j-\tau_i)\partial^2}(x+\xi_i,y+\xi_j) \ind{\tau_j > \tau_i} + J_{Airy}(i,x+\xi_i; j, y+\xi_j) \\
&= K_{Airy}(i,x+\xi_i+\tau_i^2;j,y+\xi_j+\tau_j^2).
\end{align*}
This completes the proof.

\subsection{Inclusion into the class $F(\alpha,\beta)$}
We provide a criterion to check the existence of suitable $\alpha,\beta$ such that $\nu^n \in F(\alpha,\beta)$.
\begin{prop} \label{prop:Fab}
    For $\nu = \{\nu_1, \ldots, \nu_n\}$, let $\mathrm{diam}(\nu) = \max_j \nu_j - \min_j \nu_j$. For $0 \leq \eta \leq \mathrm{diam}(\nu)$,
    let $\rho(\nu,\eta) = \frac{| \{j:\nu_j \geq \max_i \nu_i - \eta\}|}{n}$. Given a sequence of point clouds $\nu^n$, set
    $$ \alpha = \liminf_n \, \sup_{\eta} \frac{\sqrt{\rho(\nu^n,\eta)} - \eta}{2}, \quad \beta = \limsup_n\, \mathrm{diam}(\nu^n) +2.$$
    If $\alpha > 0$ and $\beta < \infty$ then $\nu^n \in F(\alpha, \beta)$ for all sufficiently large $n$.
\end{prop}

\begin{proof}
    Let $b_n = b(\nu^n)$, $\nu^n_{max} = \max_j \nu_j^n$ and $\nu^n_{min} = \min_j \nu^n_j$.
    
    Firstly we claim that $b_n \leq \nu^n_{max} + 1$. Indeed, if not, then $b_n-\nu^n_j > \nu^n_{max}-\nu^n_{max} +1 = 1$ for every $j$.
    Therefore, $(b_n - \nu^n_j)^2 > 1$ for every $j$, which shows that the average
    $$ \frac{1}{n} \sum_{j=1}^n \frac{1}{(b_n-\nu^n_j)^2} < 1,$$
    which is a contradiction. Consequently, $b_n - \nu^n_j \leq \nu_{max}+1-\nu_{min} = \mathrm{diam}(\nu^n)+1$ for every $j$.
    Therefore, $b_n - \nu^n_j \leq \beta$ for every $j$ and all sufficiently large values on $n$.

    Next, suppose $\nu^n_j \geq \nu^n_{max} - \eta$. Then, $0 < b_n-\nu^n_j \leq b_n-\nu^n_{max}+\eta$, which implies
    that $(b_n-\nu^n_j)^2 \leq (b_n-\nu^n_{max}+\eta)^2$. Therefore,
    $$1 = \frac{1}{n} \sum_{j=1}^n \frac{1}{(b_n-\nu^n_j)^2} \geq \frac{\rho(\nu^n,\eta)}{(b_n-\nu^n_{max}+\eta)^2}.$$
    This implies $b_n-\nu_{max}+\eta \geq \sqrt{\rho(\nu^n,\eta)}$. Consequently, $b_n - v^n_j \geq b_n - \nu^n_{max} \geq \sqrt{\rho(\nu^n,\eta)}-\eta$.
    Optimising over $\eta$ shows
    $$ b_n-\nu^n_j \geq \sup_{\eta} \sqrt{\rho(\nu^n,\eta)} - \eta \quad \text{for every}\; j.$$
    As a result, for all large values of $n$, $b_n-\nu^n_j \geq \alpha$ for every $j$.
\end{proof}

Suppose $H_n$ is a sequence of Hermitian matrices with eigenvalues $\lambda^n = \{\lambda^n_1, \ldots, \lambda^n_n\}$. Observe that $$\mathrm{diam}(\lambda^n) \leq 2 ||H_n||_{op}.$$
Let $$\mu_n = \frac{1}{n} \sum_{i=1}^n \delta_{\lambda^n_i}$$
be the empirical measure of the eigenvalues of $H_n$. Suppose $\mu_n$ converges weakly to a measure $\mu_{\infty}$.
Denote by $\nu_{max} = \sup \, \{ x \in \R: x \in \mathrm{support}(\mu_{\infty})\}$ the maximal point in the support of $\mu_{\infty}$ and assume it is finite.
Let $$\rho_{\infty}(\eta) = \mu_{\infty}([\nu_{max}-\eta, \nu_{max}])$$
for $\eta \geq 0$. It is easy to see that
$$ \liminf_n \, \sup_{\eta \geq 0} \, \sqrt{\rho(\lambda^n,\eta)} -\eta \geq \sup_{\eta \geq 0} \, \sqrt{\rho_{\infty}(\eta)}-\eta.$$ Thus, if we set
$$ \alpha = \sup_{\eta \geq 0}\, \frac{\sqrt{\rho_{\infty}(\eta)}-\eta}{2}, \quad \beta = 2 \limsup_n ||H_n||_{op}+2$$
then $\lambda^n \in F(\alpha,\beta)$ for all large values of $n$ provided that $\alpha > 0$ and $\beta < \infty$.

\section{Proofs of Theorem \ref{thm:Piflat} and Corollary \ref{cor:loe}} \label{sec:runningmax}
We begin with the proof of Theorem \ref{thm:Piflat}.

\begin{proof}
    Let $\mu_{max} = \max_i \mu_i < 0$ and $\mu_{min} = \min_i \mu_i$. Let $\gamma_{rec}$ be a rectangular contour that intersects the real axis at the points $\mu_{max}/2$ and $\mu_{min}-1$ and has imaginary part equal to $\pm \mu_{max}/2$ along the horizontal sides. Decompose the kernel $K_{flat}$ in \eqref{eqn:Kflat} as the sum, $(I) + (II)$,
    of the two contour integral terms. We have to consider the kernel with parameters $m=n$, $t_1=t_2 = t$ and $x,y \geq \max\{a,0\}$ in the limit $t \to \infty$.

    In the term $(I)$, choose the contour $\gamma$ to be $\gamma_{rec}$ and the contour $\Gamma$ to be the vertical line $\Re(z)=0$. Let us bound the integrand of $(I)$.
    We have $|z-w| \leq -\mu_{max}/2$ and $\prod_{=1}^n|w-\mu_i| \geq (\min\{1,-\mu_{max}/2\})^n$. Consider the modulus of $e^{\frac{t}{2}w^2-xw}$.
    If $w = u+\mathbf{i}v \in \gamma_{rec}$ then $\Re(w^2) = u^2-v^2 \geq 0$. Since $x \geq \max\{a,0\}$, $\Re(-xw) =-xu \geq x (-\mu_{max}/2)$.
    We deduce that the integrand that depends on the $w$-variable is bounded from above in modulus by
    $e^{x\mu_{max}/2} \times C_{\mu_i}$ for some constant $C$ that depends on $\mu_i$. Consider the integrand in the $z$-variable. We find that
    $\Re(tz^2/2-yz) = - tv^2/2$ if $z = \mathbf{i}v$ for $v \in \R$. Also, $\prod_{i} |z-\mu_i| \leq C_{\mu_i}|v|^n$. As the contour $\gamma_{rec}$ has bounded length, we find that
    $$ |(I)| \leq C_{\mu_i} e^{x\mu_{max}/2} \int_{-\infty}^{\infty} dv\, e^{- \frac{t}{2}v^2 + m \log(|v|+1)} = e^{x\mu_{max}/2} \epsilon_t$$
    where $\epsilon_t \to 0$ as $t \to \infty$ by the dominated convergence theorem.

    In the term $(II)$ choose $\gamma$ to be $\gamma_{rec}$ again. Shift the contour $\Gamma$ to the vertical line $\Re(z) = 0$. In doing so we encounter a simple pole as $z=-w$
    with residue
    $$\frac{1}{2 \pi \mathbf{i}} \oint_{\gamma_{rec}}dw\, e^{(x+y)w} \prod_{i=1}^n \frac{\mu_i+w}{\mu_i-w}.$$
    Changing variables $w \mapsto -w$ gives the kernel $K(x,y)$ in the statement of the theorem.
    Note that we also have the bound
    $$|K(x,y)| \leq C_{\mu_i} e^{(x+y)\mu_{max}/2}.$$

    The remainder of the term $(II)$ is a double contour integral that looks like $(I)$ except it carries the term $z+w$ instead of $z-w$.
    By the same argument as before, it is bounded in modulus by $e^{x\mu_{max}/2} \epsilon_t$.

    We conclude that the kernel $K_{flat}(t,x;t,y) \to K(x,y)$ pointwise as $t \to \infty$. The exponential decay in the parameter $x$ and the boundedness in $y$
    also ensure, by Hadamard's inequality and the dominated convergence theorem, that the Fredholm determinant of $K_{flat}(t,x,t,y)$ converges to the one of $K$
    as $t \to \infty$.   
\end{proof}

In order to establish Corollary \ref{cor:loe}, we need the following proposition. It gives a determinant formula for the law of the running maximum of the top path among $n$ noncolliding Brownian bridges.
\begin{prop} \label{prop:bridgelaw}
    Let $s \in [0,1]$ and $a > 0$. Let $B^{br}_1$ be the top path among $n$ Brownian bridges conditioned not to collide as in Section \ref{sec:bridges}. Then,
    $$ \mathbf{Pr}(\max_{t \in [0,s]}\, B^{br}_1(t) \leq a) = \det(I - K)_{L^2[0,\infty)}$$
    where the kernel $K$ equals
    $$K(x,y) = K_{flat}(a^2 s/(1-s), x+a^2; a^2 s/(1-s), a^2+y).$$
    and the corresponding drifts $\mu_i = -1$ for $1\leq i\leq n$.
\end{prop}

\begin{proof}
    For $T \geq 0$ consider the probability of the event $\max_{0 \leq t \leq a^2 T}\, \lambda_{max}(H(t)-tI) \leq a^2$.
    Since $\lambda_{max}(H(t)-tI) = \lambda_{max}(H(t))-t$, and $H(t)$ has the same law as $\alpha H(t/\alpha^2)$ for every $\alpha > 0$ (by Brownian scaling),
    we find that
    \begin{align*}
        \mathbf{Pr}\left( \max_{0 \leq t \leq a^2 T}\, \lambda_{max}(H(t)-tI) \leq a^2 \right) &=
        \mathbf{Pr}\left( \max_{0 \leq t \leq a^2 T}\, a \lambda_{max}(H(t/a^2)) \leq a^2 +t\right)\\
        &= \mathbf{Pr}\left( \max_{0 \leq t \leq T}\, a \lambda_{max}(H(t)) \leq a^2(1+t)\right)\\
        &= \mathbf{Pr}\left( \max_{0 \leq u \leq T/(1+T)}\, \lambda_{max}((1-u)H(u/(1-u))) \leq a\right)\\
        &= \mathbf{Pr}\left( \max_{0 \leq u \leq T/(1+T)}\, B^{br}_1(u) \leq a\right).
    \end{align*}
    We choose $T = s/(1-s)$ in which case $T/(1+T) = s$. It follows that
    \begin{equation} \label{eqn:bridgeidentity}
    \mathbf{Pr}\left( \max_{0 \leq u \leq s}\, B^{br}_1(u) \leq a\right) = \mathbf{Pr}\left( \max_{0 \leq t \leq a^2s/(1-s)}\, \lambda_{max}(H(t)-tI) \leq a^2 \right).
    \end{equation}
    The proposition now follows from \eqref{eqn:eigenidentity} and Proposition \ref{prop:flat}
\end{proof}

Looking at \eqref{eqn:bridgeidentity} for $s=1$ we conclude that
$$ \left (\max_{0 \leq u \leq 1}\, B^{br}_1(u) \right)^2 \stackrel{law}{=}
\max_{t \geq 0}\, \lambda_{max}(H(t)-tI).$$
Corollary \ref{cor:loe} now follows from Theorem \ref{thm:Piflat} for the case when every $\beta_i = 1$
together with \eqref{eqn:fw} and \eqref{eqn:NR}.

\end{document}